\documentclass[12pt]{article}
\usepackage{amssymb}
\usepackage{graphicx}
\usepackage{amsfonts}
\usepackage{graphicx}
\usepackage{amsmath}
\usepackage{amsthm}
\usepackage{amssymb}
\usepackage{color}

\newtheorem{thm}{Theorem}
\newtheorem{cor}[thm]{Corollary}
\newtheorem{lem}[thm]{Lemma}
\newtheorem{prop}[thm]{Proposition}
\theoremstyle{definition}

\theoremstyle{remark}
\newtheorem{rem}[thm]{Remark}
\newtheorem{step}{Step}

\newcommand{\norm}[1]{\left\Vert#1\right\Vert}
\newcommand{\abs}[1]{\left\vert#1\right\vert}
\newcommand{\set}[1]{\left\{#1\right\}}
\newcommand{\Comp}{\mathbb C\,}
\newcommand{\Real}{\mathbb R}

\newcommand{\Prob}{\mathbb P}
\newcommand{\Exp}{\mathbb E}
\newcommand{\pr}{\mathbb P}
\newcommand{\ev}{\mathbb E}
\newcommand{\Tr}{\mathrm{Tr}}

\newcommand{\Z}{\mathbb Z}

\newcommand{\eps}{\varepsilon}
\newcommand{\To}{\rightarrow}

\newcommand{\al}{\alpha}
\newcommand{\la}{\lambda}
\newcommand{\La}{\Lambda}

\newcommand{\cA}{\mathcal{A}}
\newcommand{\cB}{\mathcal{B}}

\newcommand{\cD}{\mathcal{D}}
\newcommand{\cE}{\mathcal{E}}
\newcommand{\cF}{\mathcal{F}}

\newcommand{\cN}{\mathcal{N}}

\newcommand{\cM}{\mathcal{M}}

\newcommand{\cW}{\mathcal{W}}

\newcommand{\cY}{\mathcal{Y}}
\newcommand{\cV}{\mathcal{V}}
\newcommand{\cZ}{\mathcal{Z}}

\newcommand{\zb}{\overline z}

\newcommand{\of}[1]{\left ( #1 \right ) }

\newcommand{\im}{\mathrm{Im}}
\newcommand{\re}{\mathrm{Re}}

\newcommand{\bin}[2]{\left (
\begin{array} {c}
#1 \\[-3pt]
#2
\end{array}
\right )}

\newcommand{\tph}{\widetilde\varphi}

\newcommand\vect[2]{
\left(\!\begin{array}{c}
#1\\
#2
\end{array}\!\right)}

\newcommand{\mat}[4]{\left( \begin{array}{cc}
#1 & #2  \\
#3 & #4  \\
\end{array} \right)}
\newcommand{\tla}{{\widetilde{\lambda}}}
\newcommand{\ula}{{\underline{\lambda}}}
\newcommand{\eqd}{\stackrel{d}{=}}

\definecolor{Red}{rgb}{1,0,0}
\definecolor{Blue}{rgb}{0,0,1}
\definecolor{Brown}{rgb}{1,0.47,0}
\newcommand{\eugene}[1]{#1}
\newcommand{\X}{{Q}}
\newcommand{\IO}{\mathcal{I}\,}
\newcommand{\RR}{\mathbb{R}}
\newcommand{\lip}{\stackrel{P}{\longrightarrow}}
\newcommand{\cd}{\Rightarrow}
\newcommand{\Fs}{\cA(D,\Real^d)} 
\newcommand{\sch}{{\sf Sch}_\tau}
\newcommand{\schs}{{\sf Sch}^*_\tau}

\usepackage{natbib}

\begin{document}

\title{The scaling limit of the critical one-dimensional random Schr\"{o}dinger operator}%

\author{Eugene Kritchevski \and Benedek Valk\'o\and B\'alint Vir\'ag}%




\maketitle

\begin{abstract}
We consider two models of one-dimensional discrete random
Schr\"{o}dinger  operators
$$(H_n\psi)_\ell =\psi_{\ell -1}+\psi_{\ell +1}+v_\ell \psi_\ell ,$$ $\psi_0=\psi_{n+1}=0$
in the cases $ v_k=\sigma \omega_k/\sqrt{n}$ and $
v_k=\sigma \omega_k/ \sqrt{k}. $ Here $\omega_k$ are
independent random variables with mean $0$ and variance
$1$.

We show that the eigenvectors are delocalized and the
transfer matrix evolution has a scaling limit given by a
stochastic differential equation. In both cases,
eigenvalues near a fixed bulk energy $E$ have a point
process limit. We give bounds on the eigenvalue repulsion,
large gap probability, identify the limiting intensity and
provide a central limit theorem.

In the second model, the limiting processes are the same as
the point processes obtained as the bulk scaling limits of
the $\beta$-ensembles of random matrix theory. In the first
model, the eigenvalue repulsion is much stronger.
\end{abstract}

\section{Introduction}

We consider two models of one-dimensional discrete random
Schr\"{o}dinger  operators given by the matrix
\begin{equation}\label{shrod1dmatrix}
H_n=\left( \begin{array}{cccccc}
v_1 & 1 &  &  &  & \\
1 & v_{2} & 1 &  & & \\
  & 1  &\ddots &\ddots & &\\
& & \ddots & \ddots &1 & \\
& & & 1 & v_{n-1} & 1 \\
& & & & 1 & v_{n} \\
\end{array} \right)
\end{equation}
in the following two cases, referred to as the {\bf
critical model} and {\bf decaying model} respectively:
\begin{equation}
v_k=\sigma \omega_k/\sqrt{n}, \qquad v_k=\sigma \omega_k/
\sqrt{k}.\label{models}
\end{equation}
Here $\omega_k$ are independent random variables with mean
$0$, variance $1$ and bounded third absolute moment.

We show that the eigenvectors are delocalized and the
transfer matrix evolution has a scaling limit given by a
stochastic differential equation. We show that in both
cases eigenvalues near a fixed bulk energy $E$ have a point
process limit.

We analyze the limiting point processes, in particular we
give bounds on the eigenvalue repulsion, large gap
probability, identify the limiting intensity and provide a
central limit theorem.

In the decaying model, the limiting processes are the same
as the point processes obtained as the bulk scaling limits
of the $\beta$-ensembles of random matrix theory. In the
critical model, the eigenvalue repulsion is much stronger.

\subsection*{The critical model}

For very small values of $\sigma$, this matrix behaves like
the discrete Laplacian -- its eigenvalues are locally close
to periodic and its eigenvectors are extended. {The
discrete measure constructed by the square of the
coordinates of the normalized eigenvector will not be
concentrated on any small set of points.} For large
$\sigma$, the matrix is close to diagonal, with eigenvalues
dropped independently at random (Poisson statistics)  and
eigenvectors are localized. The goal of this paper is to
examine the nature of the transition from extended to
localized eigenvectors, and the corresponding eigenvalue
statistics.

The matrix $H_n $ is a perturbation of the adjacency matrix
of a 1-dimensional box.  When  the variance of $v_k$ does
not depend on $n$, eigenvectors are localized \citep{KCM,
KuSo, GMP} and the local statistics of eigenvalues are
Poisson \citep{Mi,Mo}. For the perturbed adjacency matrix
of higher-dimensional boxes, localization \citep{AM,FS} and
Poisson eigenvalue statistics \citep{Mi} hold if the
variance is a sufficiently large constant. In dimensions
three and higher, for a small constant variance, it is
widely conjectured that one gets random-matrix type
statistics of eigenvalues and extended eigenfunctions,
while  for two dimensions the opinions vary.

Our regime, where the variance of the random variables
$v_\ell$ are of order $n^{-1/2}$ captures the transition
between localization an delocalization. We will use the
methods developed in \cite{VV} to analyze the asymptotic
local spectral properties of $H_n$.

If there is no noise (i.e.~$\sigma=0$) then the eigenvalues
of the operator  are given by $2\cos(\pi k/(n+1))$ with
$k=1,\dots, n$. The asymptotic density near $E\in
(-2,2)$ is given by $\frac{\rho}{2\pi}$ with
\begin{equation}\label{defrho}
\rho=\rho(E)=1/\sqrt{1-E^2/4}
\end{equation}
which suggests that one should scale by $n$ near $E$ to get
a meaningful limit. Thus we will study the spectrum
$\Lambda_n$ of the scaled operator
\begin{equation}\label{scaledHn}
\rho n (H_n-E).
\end{equation}

We will use the well-known transfer matrix description of
the spectral problem for $H_n$. The one-dimensional
eigenvalue equation $H_n\psi= \mu \psi$ is written as
\begin{equation}\label{txmat_form}
\bin{\psi_{\ell +1}}{\psi_{\ell }}=T(\mu-v_\ell
)\bin{\psi_{\ell }}{\psi_{\ell -1}}=M_\ell
\bin{\psi_{1}}{\psi_{0}},
\end{equation}
where $$T(x):=\mat{x}{-1}{1}{0} \textrm{ and } M_\ell
:=T(\mu-v_\ell )T(\mu-v_{\ell -1})\cdots T(\mu-v_1).$$
Then $\mu$ is an eigenvalue of $H_n$ if and only if
\begin{equation}\label{ev_cond}
M_n\bin{1}{0}=c \bin{0}{1},
\end{equation}
for some $c\in\Real$, or, equivalently $(M_n)_{11}=0$. In
view of (\ref{scaledHn}) we parametrize
$\mu=E+\frac{\la}{\rho n}$. We will use the notation
$M_\ell ^\la$ to emphasize dependence on $\la$, and use the
similar notation for other quantities.  Setting
\begin{equation}\label{epsn}
\eps_\ell =\frac{\la}{\rho
n}-\frac{\sigma\omega_\ell}{\sqrt{n}},
\end{equation}
we have
\begin{equation}\label{Mn}
M_\ell ^\la=T(E+\eps_\ell )T(E+\eps_{\ell -1})\cdots
T(E+\eps_1) \textrm{ for } 0\leq \ell \leq n.
\end{equation}
The scaling of $v_\ell=\sigma\omega_\ell/\sqrt{n}$ ensures
that, with high probability, the transfer matrices
$M_\ell^\la$ are bounded and the eigenfunctions are
delocalized. \eugene{
\begin{thm}\label{delocalization} Given $E\in (-2,2)$, $R<\infty$ and $\sigma<\infty$, there exists
a constant $c$ so that for all sufficiently large $n$
and all $t>0$, the following two statements hold with
probability at least $1-c/t$.
\newline\noindent {\rm (1)} We have
\begin{equation}\label{boundSn}
\max_{0\leq \ell\leq n,\abs{\la}\leq R}\Tr {M_\ell
^\la}{M_\ell ^\la}^* < t.
\end{equation}
\newline\noindent {\rm (2)} For each eigenvector $\psi$ of $H_n$, normalized by $\sum_{\ell=1}^{n}|\psi_\ell|^2=1$ and corresponding to an eigenvalue $\mu\in [E-\frac{R}{\rho n},E+\frac{R}{\rho n}]$, we have
\begin{equation}\label{boundpsi}
 \frac{2}{(n+1)t^2}<\abs{\psi_\ell}^2+\abs{\psi_{\ell+1}}^2 < \frac{2t^2}{n+1},\qquad 0\leq \ell\leq n.
 \end{equation}
\end{thm}
}

In order to understand the interaction of the eigenvalues
near $E$ and of the corresponding eigenvectors, one would
ideally like to derive a limiting diffusion process for
\eqref{Mn}. The starting observation is that $M_\ell ^\la$
cannot have a continuous limit.  The obstacle is that for
large $n$ each transfer matrix $T(E+\eps_k)$ in \eqref{Mn}
is not close to $I$ but to $T(E)$. Thus we are led to
consider, instead of $M_\ell ^\la$, the matrices
\begin{equation}\label{Xn}
\X_\ell ^\la=T^{-\ell }(E)M_\ell ^\la,\qquad 0\leq \ell\leq
n,
\end{equation}which will evolve regularly.
\begin{figure}[htp]
\begin{center}
\includegraphics[width=300pt]{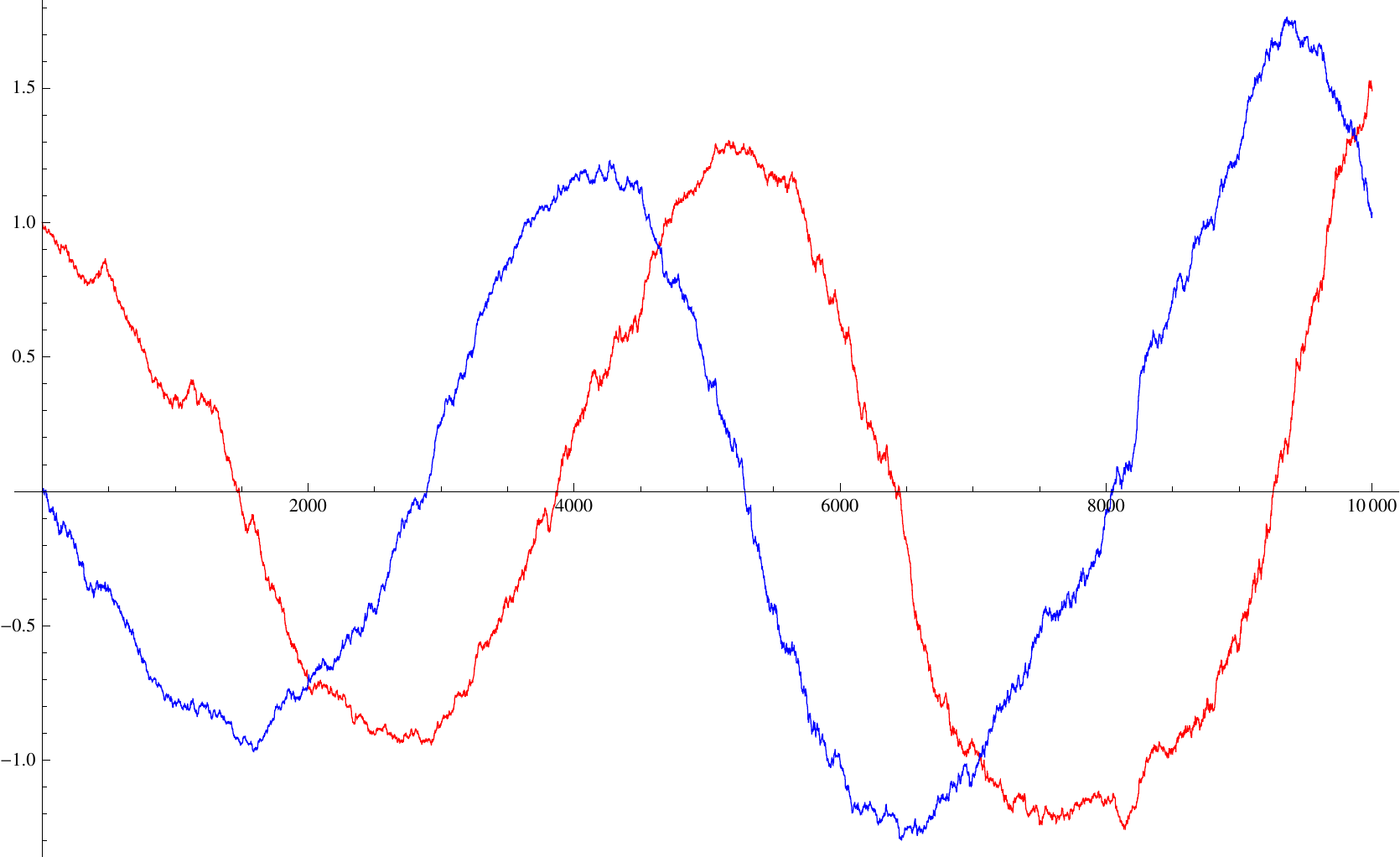}
\end{center}
\caption{Simulation of the entries of the first row of $\X_\ell ^\la$
with $E=1, \la=25$ and $n=10000$. }
\end{figure}



To control the correction factor $T^{-\ell }(E)$, we
diagonalize $T(E)=ZDZ^{-1}$ with
\begin{equation}
D=\mat{\zb}{0}{0}{z}, \quad Z=\mat{\zb}{z}{1}{1},\quad
z=E/2 + i\sqrt{1- (E/2)^2}.\label{zdef}
\end{equation}
\begin{thm}\label{DiffusionTransfer} Assume $0<|E|<2$.
Let $\cB(t), \cB_2(t),\cB_3(t)$ be independent standard
Brownian motions in $\Real$,
$\cW(t)=\frac1{\sqrt{2}}(\cB_2(t)+i \cB_3(t))$. Then the
stochastic differential equation
\begin{equation}\label{LimitingTransfer}
d\X^{\la}=\frac12  Z \of{ \mat{i \la}{0}{0}{- i \la} dt+
\mat{i d\cB}{ d\cW}{ d\overline\cW}{-i
d\cB}}Z^{-1}\X^{\la}, \qquad \X^{\la}(0)=I
\end{equation}
has a unique strong solution $\X^{\la}(t): \la\in\Comp,
t\ge 0$, which is analytic in $\lambda$. Moreover with
$\tau=(\sigma \rho)^2$
\begin{equation*}
(\X_{\left \lfloor nt/\tau \right \rfloor }^{\la}, 0\leq
t\leq \tau)\Rightarrow (\X^{\la/ \tau}(t),0\leq t\leq \tau),
\end{equation*}
in the sense of finite dimensional distributions for $\la$
and uniformly in $t$. Also, for any given $0\le t\le \tau$
the random analytic functions $\X_{\left \lfloor nt/\tau
\right \rfloor }^{\la}$ converge in distribution to
$\X^{\la/\tau}(t)$ with respect to the local uniform topology.
\end{thm}

Theorem \ref{DiffusionTransfer} is a one-dimensional
version of a more general quasi-one-dimensional theorem
that appears in \cite{VV3}. This  proof, which predates the
one in that paper, is included here for completeness. The
preprint \cite{VV3} was followed by the preprint of
\cite{BD}, who, in independent work, also study SDE limits
of transfer matrices. Their starting point the so-called
DMPK theory in the physics literature, which is essentially
the study of diffusive limits of quasi-one-dimensional
random Schr\"odinger operators from a slightly different
point of view. We refer the reader to \cite{BD} for a
discussion of this theory. One of the novelties of our
approach is that it allows for studying the dependence on
the eigenvalue $\lambda$, which in turn allows us to deduce
the scaling limit of the spectrum, the main focus here.

\bigskip

The introduction of $T^{-n}(E)$ in Theorem
\ref{DiffusionTransfer} has the effect of changing the
boundary condition for each $n$, so for the next result, we
have to pass to subsequences.

\begin{cor}\label{ConvergenceOfPoint}
Suppose that  $n_j$ is a subsequence so that $z^{n_j}$
converges.
 Then $T^{n_j}(E)\to \tilde T$ and the random matrix-valued analytic functions
$M_{n_j}^\lambda$ converge in distribution to $\tilde T
Q^{\lambda/\tau}(\tau)$. Moreover, $\La_{n_j}$ converges in
law to the counting measure of the zeros of the random
analytic function $\lambda\mapsto [\tilde T
Q^{\lambda/\tau}(\tau)]_{11}$.
\end{cor}

Note that the sequence $M_{n_j}$ has a limit if $z^{n_j}$
converges. For $\Lambda_{n_j}$ to have a limit, only the
convergence of $z^{2n_j}$ is needed. If it does, then the
possible limits of $T^{n_j}$ are $\pm \tilde T$ and
$[\tilde T Q(\tau)]_{11}$ and $-[\tilde T Q(\tau)]_{11}$
have the same zero set.

Since $H_n$ is symmetric, the limiting point process will
live on $\Real$. The  point process can be more effectively
described by a scalar SDE. We first note that for any
$a,b\in \Real^2$ we have
\[
Z^{-1} \vect{a}{b}=\frac{\rho}{2}\vect{
 a i-b z i}{\overline{a i-b z i}
},
\]
so $Z^{-1}$ maps real vectors to vectors with conjugate
entries. Since for $\lambda\in \Real$ the transfer matrix
$Q^\la_\ell$, and the limiting process $Q^\la(t)$ will also
be real, we can write
\[
%
Z^{-1}Q^\la(t)\vect{1}{0}=\vect{i\psi^\la(t)}{\overline{i
\psi^\la(t)}}, \qquad
\]
for some complex numbers $\psi^\la(t)$ where
$\psi^\la(0)=\rho/2$ (the extra $i$ in the above definition
makes this and some upcoming formulas nicer). We will
define the \emph{phase function} $\varphi^\la(t)$ by
\begin{equation}\label{defphi}
e^{i\varphi^\la(t)}= {\psi^\la(t)/\overline{\psi^\la(t)}},
\qquad \varphi^\la(0)=0.
\end{equation}
This uniquely determines $\varphi^\la(t)$ assuming that it
is continuous in $t$ (as long as $\det Q^\la(t)\neq 0$,
which follows from  (\ref{LimitingTransfer})). It\^o's
formula then gives an SDE for the evolution of
$\varphi^\la(t)$ and we can identify the zeros of $[\tilde
TQ(\tau)]_{11}$. This leads to another description of the
point process limit of $\Lambda_{n_j}$.

\begin{figure}[htp]
\begin{center}
\includegraphics[width=300pt]{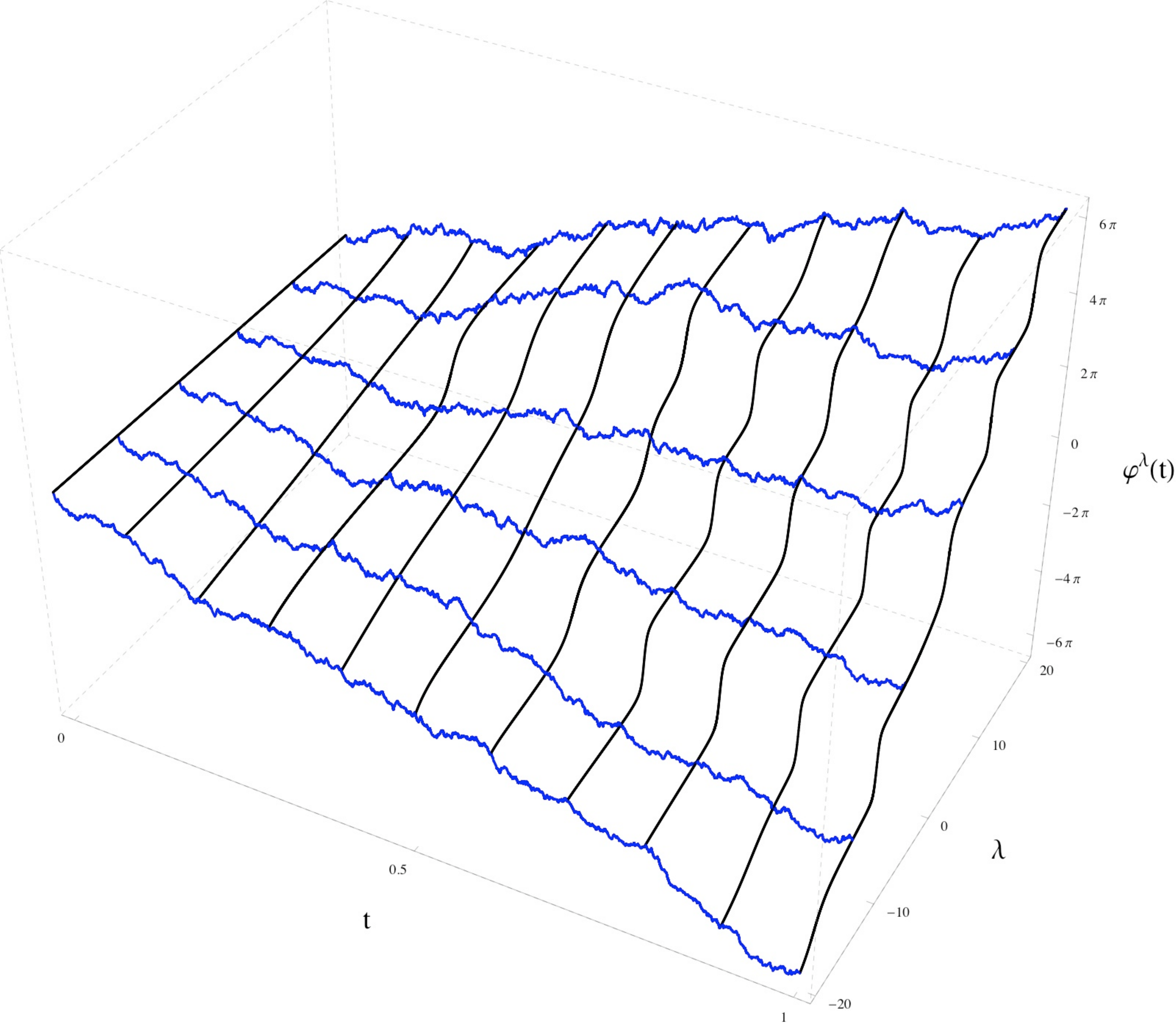}
\end{center}
\caption{The phase function $\varphi^\lambda(t)$ for
$(t,\lambda)\in[0,1]\times[-20,20]$.}
\end{figure}

\begin{cor}[Schr\"odinger random
analytic functions]\label{limitPointProcess} Consider the
family of SDE's
\begin{equation}\label{SDEphase}
d\varphi^\la(t)=\la dt+ d\cB+\re\left[
e^{-i\varphi^\la(t)}d\cW \right],\quad \varphi^\la(0)=0
\end{equation}
coupled together for all values of $\lambda\in \Real$ where
$\cB$ and $\cW$ are standard real and complex Brownian
motions.  This has a unique strong solution and for each
time $t$ the function $\la\mapsto \varphi^\la(t)$ is
strictly increasing and real-analytic with probability one.

Moreover, for  $0<|E|<2$ and with $\tau=(\sigma \rho)^2$
the point process $\Lambda_n-\arg(z^{2n+2})-\pi$ converges 
in distribution to the point process
\begin{equation}\label{e:schdef}
\sch:= \set{\lambda: \varphi^{\lambda/\tau}(\tau)\in 2 \pi
\Z}.\end{equation}
\end{cor}
\begin{rem}\label{rem5}
\noindent Note that the SDE (\ref{SDEphase}) for a
\emph{fixed} $\lambda$ describes a Brownian motion with
variance $\frac{3}{2}$ and drift $\lambda$. The random
analytic function $\lambda\mapsto \varphi^\lambda(\tau)$ is
given by the values of these coupled drifted Brownian
motions at time $\tau$. In Lemma \ref{invariance} we show
that $\varphi^{\la/\tau}(\tau)+\theta\eqd
\varphi^{(\la+\theta)/\tau}(\tau)$ which means that
$\sch+\theta\eqd\set{\lambda: \varphi^{\lambda/\tau} \in
\theta+2 \pi \Z}$.
\end{rem}

Note that the point process $\sch$ is invariant under
translation by integer multiples of $2\pi$, but not under
other translations. To fix this, we consider a translation
by an independent uniform random variable:
$$
\schs=\sch+U[0,2\pi].
$$
This version can be described through a variant of the the
Brownian carousel introduced in \cite{VV} (the same is true
for $\sch$, but with more complicated boundary conditions).

\paragraph{The Brownian carousel.} Let $(\mathcal V(t),t\ge 0)$ be Brownian motion on the hyperbolic
plane $\mathbb H$. Pick a point on the boundary $\partial
\mathbb H$ and let $x^\lambda(0)$ equal to this point for
all $\lambda\in \mathbb R$. Let $x^\lambda(t)$ be the
trajectory of this point rotated continuously around
$\mathcal V (t)$ at speed $\lambda$. Recall that Brownian
motion in $\mathbb H$ converges to a point $\mathcal
V(\infty)$ in the boundary $\partial \mathbb H$.

\begin{thm}[Brownian carousel
description]\label{thm:carousel} We have
\[\{\lambda:  x^{\lambda/\tau}(\tau)=\mathcal V(\infty)\} \;
\stackrel{d}{=}\; \schs.
\]
\end{thm}

Section \ref{s:carousel} contains the proof of Theorem
\ref{thm:carousel} and a description of the ODE for
 $x^\lambda(t)$ in the Poincar\'e disk model of the
hyperbolic plane. Amazingly, a (less complete) connection
between random Schr\"odinger operators and Brownian motion
in the hyperbolic plane had been found already in \cite{GeVa}.
\bigskip

Note that in the previous theorems we assumed $0<|E|<2$.
The case $E=0$ is slightly different, but it gives similar
results. Note that in that case $z=i$.
\begin{thm}\label{theoremE=0}
Let $E=0$. In that case Theorem \ref{DiffusionTransfer},
and Corollaries \ref{ConvergenceOfPoint} and
\ref{limitPointProcess} hold with the following SDE's in
place of (\ref{LimitingTransfer}) and (\ref{SDEphase}):
\begin{equation}\label{LimitingTransfer0}
d\X^{\la}=\frac12  Z \of{ \mat{i \la}{0}{0}{- i \la} dt+
\mat{i d\cB_1}{ i d\cB_2}{ - i d\cB_2}{-i
d\cB_1}}Z^{-1}\X^{\la}, \qquad \X^{\la}(0)=I.
\end{equation}
and
\begin{equation}\label{SDEphase0}
d\varphi^\la=\la dt+ d\cB_1+ \cos(\varphi^\la)
d\cB_2+\frac{1}{4} \cos(2\varphi^\la) dt,\quad
\varphi^\la(0)=0.
%
%
%
\end{equation}
\end{thm}
Note that since $z=i$ we just need to fix the remainder of
$n_j \textup{ mod }4$ for $z^{n_j}$ to converge and the
parity of $n_j$ for $\arg(z^{2n_j+2})$ to converge.

\subsection*{Properties of the limiting eigenvalue process for the critical model}

We now discuss some of the properties of the limit process
$\sch$. We describe the eigenvalue repulsion,  we compute
the intensity of the point process and then give the
asymptotic probability of finding a large gap. We also
provide a central limit theorem for the number of points in
a growing interval. Let $\sch[a,b]$ denote the number of
points of $\sch$ in the set $[a,b]$.

\begin{thm}[Eigenvalue repulsion]\label{repulsion} For $\mu\in \Real$ and $\eps>0$
we have
\begin{equation}\label{bound2eig}
\Prob\set{\sch[\mu,\mu+\eps]\geq 2}\leq 4\exp\left
(-(\log(\tau/\eps)-\tau)^2/\tau\right).\end{equation}
whenever the squared expression is nonnegative.
\end{thm}

\begin{rem}
In case of the classical random matrix models GOE, GUE, GSE
the eigenvalue repulsion is a lot weaker: it is of the
order of $\eps^{2+\beta}$ where $\beta=1, 2$ and $4$ in the
respective cases.
\end{rem}

\begin{thm}[Intensity of the point process]\label{intensity} The intensity measure $A\mapsto
\Exp\,\sch(A)$ has density
\[
\sum_{k} p(2\pi k+x), \qquad p(y)=\frac{1}{\sqrt{3 \pi\tau
}}e^{-\frac{1}{3\tau} y^2}
\]
at $x$. This is the density of a centered normal random
variable with variance $\frac{3}{2}\tau$ mod $2\pi$ (a
theta function).
\end{thm}
\begin{rem}
If the random variables $\omega_\ell $
have a a bounded probability density $g(x)dx$, then the
general Wegner and Minami's estimates for random discrete
Schr\"{o}dinger operators \citep{Mi,GV,BHS,CKG} give
\begin{equation}\label{wegner}
\Prob(\La_n[\mu_1,\mu_2]\geq 1)\le
\sqrt{n}\sigma^{-1}\|g\|_\infty  \of{\mu_2-\mu_1},
\end{equation}
and
\begin{equation}\label{minami}
\Prob\of{\La_n[\mu_1,\mu_2]\geq 2}\leq \frac{\pi^2}{2}
n\sigma^{-2}\|g\|_\infty^2 \of{\mu_2-\mu_1}^2.
\end{equation}
Since we rescale the potential by $\sqrt{n}$, both
\eqref{wegner} and \eqref{minami} diverge as $n\To\infty$,
whereas Theorems \ref{repulsion} and \ref{intensity} give
effective bounds. Moreover, this theorem applies to
singular potentials, such as Bernoulli random variables
$\pm 1$ with probability $1/2$.
\end{rem}
\begin{thm}[Probability of large gaps] \label{gaps} The probability that $\sch$ has a large gap is \[
 \Prob
(\sch[0,\lambda]=0)=\exp\left\{-\frac{\lambda^2}{4\tau}(1+o(1))\right\}
\]
where $o(1)\to 0$ for a fixed $\tau$ as $\lambda\to
\infty$.
\end{thm}
The following theorem shows that for large $\lambda$ the
number of points in $[0,\lambda]$ is close to a normal
random variable with mean $\lambda/(2\pi)$ and a
\emph{fixed} variance.
\begin{thm}[Central limit theorem]\label{thmclt}
As $\lambda\to \infty$ we have
\[\of{\varphi^{0}(\tau),\varphi^{\la}(\tau)-\la}\Rightarrow (\xi_0+\xi_1,\xi_0+\xi_2)
\]
where $\xi_0,\xi_1,\xi_2$ are independent mean zero normal
random variables with variances $\tau,\tau/2, \tau/2,$
respectively.

In particular, for $\theta\in[0,2\pi)$ and $k\to\infty$
along the integers we have
\begin{equation}\label{cltformula}
\sch[0,2\pi k+\theta]-k\Rightarrow
\left\lfloor\frac{\xi_0+\xi_2+\theta}{2\pi}\right\rfloor-\left\lfloor\frac{
\xi_0+\xi_1}{2\pi}\right\rfloor.
\end{equation}
\end{thm}
%

\subsection*{The decaying model}

The decaying model \eqref{models}  can be thought of as the truncation of a discrete Schr\"odinger operator on the infinite half line with potential $v_k=\sigma \omega_k/\sqrt{k}$.
Similar operators have been studied in the literature
\citep[see][and references therein for earlier
works]{DSS,KiselevLastSimon}.
In these works, the standard deviation of the $k$th
diagonal element is on the order of $k^{-\al}$ for
$\alpha>0$. Depending on $\al$, the operator has different
generic spectral properties.
\begin{itemize} \item Slow decay: for
$0<\al<1/2$, the spectrum is pure point with
probability one. \item Fast decay: for $\al>1/2$, the
spectrum  is absolutely continuous with
probability one. \item Critical decay: for $\al=1/2$, and
small enough $\sigma$, the spectrum  is
singular continuous on an interval $(-c,c)$ and pure point
on $(-2,-c)\cup (c,2)$, with probability one.
\end{itemize}

It is natural to investigate the fine eigenvalue statistics
in these three cases. Motivated by this question, \cite{KS}
described the local behaviour of the spectrum in the
context of random CMV matrices, the unitary analog of
one-dimensional discrete Schr\"{o}dinger operators.

Our decaying model corresponds to the critical case. It
will be convenient to reverse the indices to have
$v_k=\sigma \omega_k/\sqrt{n+1-k}$.
We scale near $E\in (-2,2)\setminus\set{0}$ and define $\X_\ell ^\lambda$ as before. This process will converge to an SDE similar to (\ref{LimitingTransfer}), but the convergence will only hold on $[0,1)$. 
\begin{thm}\label{thmdecay1} We have the following limit on $[0,1)$:
\begin{equation}\label{LimitingTransferDecay}
d\X^{\la}= \frac12 Z\of{  \mat{i \la}{0}{0}{-i \la}dt +
\frac{\sigma\rho}{\sqrt{1-t}} \mat{ i d\cB}{ d\cW}{
d\overline\cW}{-i d\cB}}Z^{-1}\X^{\la}, \qquad
\X^{\la}(0)=I.
\end{equation}
in the sense of finite dimensional distributions for $\la$
and uniformly on compacts in $t$.
\end{thm}
\noindent From (\ref{LimitingTransferDecay}) and It\^o's formula it follows  that the phase function $\varphi^\lambda$ (which can be defined the same way as in \eqref{defphi}) will satisfy the following SDE on $[0,1)$:
\begin{equation}\label{SDEphaseDecay}
d\varphi^\la(t)={\la}dt+\frac{\sigma\rho}{\sqrt{1-t}}
\of{d\cB+\re\left[ e^{-i\varphi^\la(t)}d\cW \right]},\quad
\varphi^\la(0)=0.
\end{equation}
Note that since $\int_0^1\frac{1}{1-t} dt=\infty$, the process $\varphi^{\lambda}(t)$ does not have a limit as $t\to 1$.  However the \emph{relative} phase function $\alpha^{\lambda}=\varphi^{\lambda}-\varphi^0$ will converge and its limit will describe the point process limit of the spectrum.
\begin{thm}\label{limitPointProcessDecayingCase}
Let $\alpha^\la(t)$, $0\leq t\leq 1$, $\alpha^\la(0)=0$, be the solution to
\begin{equation}\label{SDErelphasedecay}
d\alpha^\la(t)={\la}dt+ \frac{\sigma\rho}{\sqrt{1-t}}
\of{\re\left[ (e^{-i\alpha^\la(t)}-1)d\cW \right]}.
\end{equation}
The function  $g(\la)=\frac1{2\pi}\lim_{t\to 1^{-}}
\alpha^\la(t)$ is integer valued and non-decreasing. If
$n\to \infty$ then the scaled eigenvalue process (see
\eqref{scaledHn}) converges to a point process $\La$ with
counting function $g$.
\end{thm}
Applying the time change $t=1-e^{-\beta \tau/4}$ for the SDE (\ref{SDErelphasedecay})  we get
\begin{equation}\label{sineb}
d \alpha^{ \lambda}=\lambda \frac{\beta}{4} e^{-\frac{\beta}{4}t} dt+\re\left[(e^{-i \alpha^{ \lambda}}-1) (d\cB_1+i d\cB_2) \right], \qquad  \alpha^{ \la}(0)=0, \quad t\in[0,\infty),
\end{equation}
where $\cB_1, \cB_2$ are independent standard Brownian motions
 and
$\beta=\frac{2}{(\sigma\rho)^2}$; this is precisely the SDE
that describes the $\textup{Sine}_\beta$ process given in
\cite{VV}.

\begin{cor}
The point process $\Lambda$ agrees with the point process
$\textup{Sine}_\beta$, the bulk limit of the beta Hermite
ensembles of random matrix theory with
$\beta=\frac{2}{(\sigma\rho)^2}$.
\end{cor}

 The Hermite $\beta$-ensemble is a finite ensemble with joint density
\[
Z_{\beta,N}^{-1} \prod_{i<j} |\lambda_i-\lambda_j|^\beta e^{-\frac{\beta}{4} \sum_{i} \lambda_i^2},
\]
this suggests that the eigenvalue repulsion is of the order
of $\eps^{2+\beta}$ (in the sense of Theorem
\ref{repulsion}), and this can be proved using
\eqref{SDErelphasedecay}. In \cite{VV} it was proved that
$\textup{Sine}_\beta$ is translation invariant with density
$(2\pi)^{-1}$ this provides the analogue of Theorem
\ref{intensity}. The asymptotic probability of large gaps
was identified in \cite{VV2}. As $\la\to \infty$ we have
\begin{equation}\label{sinebgap}
\Prob \of{\textup{Sine}_\beta[0,\la]=0}=(\chi_\beta+o(1))\lambda^{\gamma_\beta}\,\exp\left(-
\frac{\beta}{64}\lambda^2+\left(\frac{\beta}{8}-\frac14\right)\lambda\right)
\end{equation}
with $\gamma_\beta=\frac{1}{4}\left(\frac\beta{2}+\frac{2}{\beta}-
3\right)$ and $0<\chi_\beta<\infty$.

We will also prove the analogue of Theorem \ref{thmclt}.
\begin{thm}\label{cltsineb}
As $\lambda\to \infty$ we have
\[
\frac1{\sqrt{\log \lambda}} \of{\textup{Sine}_\beta[0,\la]-\frac{\lambda}{2\pi}} \Rightarrow \cN(0,\frac{2}{\beta \pi^2}).
\]
\end{thm}
An $n\to\infty$ version of this theorem for finite matrices
from circular and Jacobi $\beta$ ensembles was shown by \cite{K}.



%

Section
\ref{s:properties} contains the proofs about the various
properties of the limiting point processes. Section
 \ref{s:carousel} will describe some connections to
the Brownian carousel introduced in \cite{VV}
and prove the theorem about the carousel representation of the limiting point process.
In Sections
\ref{s:SDEconvergence} and \ref{s:pointprocess} we will
provide the proofs for the scaling limit of the spectrum
for the first model (with the constant variance) together with the delocalization of eigenvectors.
Section \ref{s:decaying} will deal with the proof in the
case of the second model (with the decaying potential).
Finally, the Appendix (Section \ref{s:appendix}) contains the proof for the existence of unique analytic solutions for the discussed SDEs and a technical proposition about the convergence of discrete time Markov chains to stochastic differential equations.

\section{Analysis of the limiting point process}\label{s:properties}

In this section we will provide the proofs for the theorems
related to various properties of our limiting point
processes.

We will first show a translation invariance property for
the phase function $\varphi$.

\begin{lem}[Invariance]\label{invariance}
For every $\theta\in \mathbb R$ we have
$$\varphi^{\la-\theta}(t)+\theta t\eqd \varphi^{\la}(t)$$ as
functions of $\lambda, t$.
\end{lem}
\begin{proof}
From (\ref{SDEphase}) it is clear that $\tilde
\varphi^\la(t)=\varphi^{\la-\theta}(t)+\theta t$ satisfies
the following one parameter family of SDEs:
\[
d\tilde \varphi^\la(t)=\la dt+ d\cB+\re\left[
e^{-i\varphi^{\la-\theta}(t)}d\cW \right],\quad \tilde
\varphi^\la(0)=0.
\]
Since $\varphi^{\la-\theta}(t)=\tilde \varphi^{\la}
(t)-\theta t$ we have $$\re\left[
e^{-i\varphi^{\la-\theta}(t)}d\cW \right]=\re\left[
e^{-i\tilde \varphi^\la(t)}d \widetilde\cW \right] ,\qquad
\mbox{where }  \widetilde \cW(t)=\int_0^t e^{i \theta s}
d\cW$$ is also a standard complex Brownian motion. Thus
$\tilde \varphi^\la$ satisfies the same SDE system as
$\varphi^\la$ with a different driving Brownian motion. The
uniqueness of solutions shown in the Appendix implies that
they indeed have the same distribution.
\end{proof}

In order to study the point process $\sch$ we will use
Corollary \ref{limitPointProcess}.
 Note that for $\lambda\in \Real$ the function
$\alpha^{\lambda}(t)=\varphi^{\lambda}(t)-\varphi^{0}(t)$
(which we will call \emph{relative phase function})
satisfies the following SDE system
\begin{equation}\label{SDErelphase}
d\alpha^\la(t)=\la dt+ \re\left[
\of{e^{-i\alpha^\la(t)}-1}d\cZ \right],\quad
\alpha^\la(0)=0,
\end{equation}
where $\cZ$ is a standard complex Brownian motion with
$d\cZ=e^{-i \varphi^0} d\cW$. For any fixed $\lambda\in
\Real$ this can be rewritten as
\begin{equation}\label{sineSDE}
d\alpha^\la=\la dt+ \sqrt{2}  \sin(\alpha^{\lambda}/2)
d\cB^\lambda,\quad \alpha^\la(0)=0,
\end{equation}
where $\cB^{\lambda}$ is a standard Brownian motion with $d\cB^{\lambda}=-\sqrt{2} \re\left[e^{-i \alpha^\la/2} d\cZ\right]$.

The quantity $\frac1{2\pi} \alpha^\la(\tau)$ gives a good
approximation for the number of points in $[0,\lambda]$.
Indeed, by Corollary \ref{limitPointProcess} we have
\begin{equation}\left|\frac1{2\pi}
\alpha^{\lambda/\tau}(\tau)-\sch[0,\lambda]\right|\le
1.\label{e:compare}\end{equation}

\begin{proof}[Proof of Theorem \ref{repulsion}
(Eigenvalue repulsion)] We will give two proofs of this
theorem. The first one uses the SDE representation of
Corollary \ref{limitPointProcess}. A second proof,  at the
end of Section \ref{s:carousel}, uses a geometric approach
via the Brownian carousel.


For $\mu=0$ by \eqref{e:compare} we have
\begin{equation}
\Prob(\sch [0,\eps]\geq 2)\; \leq\; \Prob
(\alpha^{\eps/\tau}(\tau)\geq 2\pi).\label{e:gapbound}
\end{equation}
and the same holds for other $\mu$, with
$\alpha^{\lambda/\tau}$ replaced by
$\varphi^{(\lambda+\mu)/\tau} -\varphi^{\lambda/\tau}$,
which satisfies the same SDE. Since this SDE is the only
thing we use we can restrict our attention to $\mu=0$.

Introduce the new process $Y=\log(\tan(\alpha/4))$ which is
well defined for $\alpha\in(0,2\pi)$. By \eqref{sineSDE}
and It\^o's formula the process  $Y$ satisfies the SDE
\begin{equation}\label{SDEY}
dY=\frac{\eps/\tau}2\cosh(Y)dt+\frac14
\tanh(Y)dt+\frac{1}{\sqrt{2}} \,dB
\end{equation}
with initial condition $Y(0)=-\infty$. It is clear that
$\set{\alpha^\eps(\tau)\geq 2\pi}=\set{Y \textup{ explodes
on $[0,1]$}}$. Consider now the solution $\tilde Y$ of the
SDE (\ref{SDEY}) with initial condition $\tilde Y(0)=0$. By
monotonicity we have
\[
Y \textup{ explodes on $[0,\tau]$}\Rightarrow \tilde Y
\textup{ explodes on $[0,\tau]$} \Rightarrow
\sup_{t\in[0,\tau]}|\tilde Y(t)|\ge \log (\tau/\eps).
\]
For  $\eps/\tau\le 1$, the inequality $|y|\le \log
(\tau/\eps)$ implies
\begin{equation*}
|{\eps/\tau}\cosh(y)/2+  \tanh(y)/4|\le 1.
\end{equation*}
This  means that for any $s\ge 0$ we have
$$\sup_{t\in[0,s]}|\tilde Y(t)|\le
\log(\tau/\eps) \quad \Rightarrow \quad \sup_{t\in
[0,s]}|\tilde Y(t)-\frac{1}{\sqrt{2}} B(t)|\le s.$$ Let $T$
be the first hitting time of $\log(\tau/ \eps)$ by $\tilde
Y$. Then by the previous argument we have
$\frac{1}{\sqrt{2}} |B(T)|\ge\log(\tau/ \eps)-T $ which
leads to
\begin{align*}
\Prob(\sup_{t\in[0,\tau]}|\tilde Y(t)|\ge \log
(\tau/\eps))&
\le \Prob(\frac{1}{\sqrt{2}}\sup_{t\in[0,\tau]}| B(t)|\ge \log(\tau/\eps)-\tau)\\
&\le 4\exp\left(-(\log(\tau/ \eps)-\tau)^2/\tau \right).
\end{align*}
Here the last inequality follows from Brownian scaling and
\[\Prob(\sup_{t\in[0,1]}| B(t)|\ge x)\le2 \Prob(\sup_{t\in[0,1]} B(t)\ge x)=4 \Prob(B(1)\ge x)\le 4e^{-x^2/2}.\qedhere\]
\end{proof}

For the proof of Theorem \ref{intensity} we need the
following estimate.
\begin{lem}\label{l:secondderivative} We have
$$\ev \left(\partial_\lambda^2 \varphi^\lambda(t)\right)^2
= f(t)<\infty.$$
\end{lem}
\begin{proof}
Differentiating (\ref{SDEphase}) twice in $\lambda$ is
justified by Theorem \ref{t:potter} in the Appendix. We get that for a
fixed $\lambda$, with primes denoting $\lambda$-derivatives
\begin{equation}\label{sde2}
d\varphi'=dt+ \varphi' d\cB_1, \quad d\varphi''= \varphi''
d\cB_1- \varphi'^2 d\cB_2, \quad \varphi'(0)=\varphi''(0)=0
\end{equation}
where $\cB_1, \cB_2$ are independent real Brownian motions
with variance 1/2. This shows that the distribution of the
first and second derivatives does not depend on $\lambda$,
as we already know from the invariance Lemma
\ref{invariance}.

From the first SDE we get $\Exp \varphi'(t)= t$. Applying
It\^o's lemma for $\varphi'^2$ and $\varphi'^4$ and then
using Gronwall's inequality gives that $\Exp \varphi'(t)^2$
and $\Exp \varphi'(t)^4$ are bounded as functions of $t$
only. It\^o's lemma applied to $\varphi''^2$ gives
\[
\Exp \varphi''(t)^2=\int_0^t \Exp (\varphi''(s)^2+
\varphi'(s)^4) ds.
\]
Gronwall's inequality and the fact that $\Exp
\varphi'^4(s)$ is bounded leads to the desired bound.
\end{proof}

\begin{proof}[Proof of Theorem \ref{intensity} (Intensity of the point process)]
We will evaluate
\begin{equation}\label{intensity1}
\tau^{-1}\lim_{\eps\to 0} \eps^{-1}
\Exp\#\set{[\varphi^{\mu}(\tau),\varphi^{\mu+\eps}(\tau))]\cap\of{2\pi
\Z}}= \lim_{\eps\to 0} (\tau\eps)^{-1} \Exp\left[\sch[\tau
\mu,\tau(\mu+\eps)] \right]
\end{equation}
By Corollary \ref{limitPointProcess} the function
$\varphi^{\lambda}(t)$ is analytic in $\lambda$.  We will
use the notation $\varpi^\lambda(t)={\partial_\lambda}
\varphi^{\lambda}(t)$ for the derivative. We will first
evaluate the limit in (\ref{intensity1}) by switching the
interval on the right with
$[\varphi^{\mu}(\tau),\varphi^{\mu}(\tau)+\eps
\varpi^{\mu}(\tau)]$. Then we will show that the error we
make is asymptotically small.

From (\ref{SDEphase}) we get that $\varpi^{\lambda}$ satisfies the following SDE:
\begin{equation}\label{derphase}
d \varpi^{\lambda}=dt+\re \left[  (- i \varpi^{\lambda})
d\cZ^\la \right]=dt+   \varpi^{\lambda} \im \left[ d\cZ^\la
\right] ,\qquad \varpi^{\lambda}(0)=0
\end{equation}
where $d\cZ^\la=e^{-i\varphi^\la} d\cW $. Since the SDE in
(\ref{SDEphase}) has the noise term $\re[d\cZ^\la]$ and the
last equation has $\im[d\cZ^\la]$, the two processes are
independent (for a given fixed $\lambda$). From the SDEs
(\ref{SDEphase}) and (\ref{derphase}) we get that
$\varphi^{\mu}(\tau)$ is $\cN(\mu\tau,\frac32 \tau )$ and
$\Exp \varpi^{\mu} =\tau$. Using the independence of
$\varphi^{\mu}(\tau)$ and $\varpi^{\mu} $ we get
\[
\tau^{-1}\lim_{\eps\to 0} \eps^{-1}
\Exp\#\set{[\varphi^{\mu}(\tau),\varphi^{\mu}(\tau)+\eps
\varpi^{\mu}(\tau)]\cap\of{2\pi \Z}}= \sum_{k\in \Z} p(2k
\pi+\mu\tau)
\]
where $p(\cdot)$ is the density of $\cN( 0,\frac32 \tau)$.
The only thing left is to show is that
\begin{equation}\label{intensity2}
\lim_{\eps\to 0} \eps^{-1}
\Exp\#\set{[\varphi^{\mu}(\tau)+\eps
\varpi^{\mu}(\tau),\varphi^{\mu+\eps}(\tau))]\cap\of{2\pi
\Z}}=0.
\end{equation}
We start by noting that if $Z$ is a random variable with density $f(x)$ and $x\in \Real, y\in \Real_+$ then
using the notation $f_{2\pi}(x)=\sum_{k\in \Z} f(x+2\pi k)$ we have
\begin{eqnarray}
\Exp \#\set{[Z+x,Z+x+y]\cap\of{2\pi \Z}}&=&\int_{x}^{x+y}
f_{2\pi}(-s) ds \le |y| \max_{s\in [0,2\pi)}
f_{2\pi}(s).\label{eq:aux}
\end{eqnarray}
The same upper bound holds if $y<0$.

By (\ref{SDEphase}) $X=\varphi^{\mu}(\tau)+\eps
\varpi^{\mu}(\tau)$ can be written as $X_0+\cB(\tau)$ where
$\cB$ is a standard Brownian motion independent of $\cW$
and $X_0$ is measurable with respect to the $\sigma$-field
generated by $\cW$. Since the density function of $
\cB(\tau)\textup{ mod } 2\pi$ is bounded by a $\tau$
dependent constant, we may use (\ref{eq:aux}) after
conditioning on $\cW$, which gives
\[
 \Exp\#\set{[\varphi^{\mu}(\tau)+\eps \varpi^{\mu}(\tau),\varphi^{\mu+\eps}(\tau))]\cap\of{2\pi \Z}}\le
c \, \Exp
\abs{\varphi^{\mu+\eps}(\tau)-(\varphi^{\mu}(\tau)+\eps
\varpi^{\mu}(\tau))}.
\]
To bound the right hand side, we use the integral form of
the remainder in the Taylor expansion
\begin{eqnarray*}
\Exp
\abs{\varphi^{\mu+\eps}(\tau)-(\varphi^{\mu}(\tau)+\eps
\varpi^{\mu}(\tau))}&\le& \eps \int_\mu^{\mu+\eps} \Exp
\abs{\partial^2_\lambda \varphi^x(\tau)} dx \le c \eps^2.
\end{eqnarray*}
In the last step we used the Cauchy-Schwarz inequality and
Lemma \ref{l:secondderivative}. This proves
(\ref{intensity2}) and completes the proof of Theorem
\ref{intensity}.
\end{proof}

\begin{proof}[Proof Theorem \ref{gaps} (Probability of large gaps)]
Let $\alpha=\alpha^{\lambda/\tau}$. We bound the desired
probability in terms of phase function events:
\begin{equation}\label{gap1}
 \Prob\of{\varphi^{0}(\tau)\in \of{0,{\eps}}
 \,\, \textup{mod}\,\, 2\pi \textup{ and } \alpha(\tau)\le  2\pi-\eps}
 \le \Prob\of{\sch[0,\la]=0}\le
 \Prob\of{\alpha(\tau)\le 2\pi}
\end{equation}
which is clear from Corollary \ref{limitPointProcess} and
the definition of $\alpha$. To get a lower bound first note
that we have $\varphi^0(\tau)=\cB(\tau)- \re\left[\cZ(\tau)
\right]$ where $\cZ$ is the driving Brownian motion in the
SDE (\ref{SDErelphase}) for $\alpha$. Since $\cB$ is
independent of $\cZ$ we have
\[
\Exp\left[   \mathbf{1}\set{\varphi^{0}(\tau)\in \of{0,
{\eps }} \,\, \textup{mod}\,\, 2\pi }    \Big|    \cZ
\right]\ge \eps \min_x f_{2\pi}(x) =\eps c_\tau>0.
\]
where $f_{2\pi}(x)$ is the density of $\cB(\tau) \textup{
mod } 2\pi$. This means that the lower bound in
(\ref{gap1}) can be estimated with $\eps c_\tau
\Prob\of{\alpha(\tau)\le 2\pi-\eps}$ from below.

Recall the SDE (\ref{sineSDE}):
\[
d\alpha=(\la/\tau) dt+ \sqrt{2} \sin(\alpha/2) d\cB,\quad
\alpha(0)=0.
\]
In Theorem 13 of \cite{VV} the authors analyze $\lim_{t\to
\infty} P(\tilde\alpha^\la(t)\le 2\pi)$ for a similar SDE:
\[
d\tilde \alpha=\lambda f dt+2 \sin(\tilde \alpha/2)
dB,\quad \tilde \alpha(0)=0
\]
and with certain weak assumptions on $f$ they get the
asymptotics $\exp(-\lambda^2 ||f||^2_2/8+o(1))$. The exact
same methods with $f=\frac{2}{\tau}{\bf  1}_{[0,
{\tau/2}]}$ in the present case give
\[
 \Prob\of{\alpha(\tau)\le (2-\eps)\pi}\ge \exp\left\{- \frac{\lambda^2}{4\tau} (1+o(1))\right\}, \qquad
\Prob\of{\alpha(\tau)\le 2\pi}\le \exp\left\{-
\frac{\lambda^2}{4\tau} (1+o(1))\right\}.
\]
We omit the straightforward details.
\end{proof}

The asymptotic gap probability for the
 $\textup{Sine}_\beta$ process (see formula
(\ref{sinebgap})) was analyzed to higher precision in
\cite{VV2}. Those techniques may also work here, resulting
in an asymptotic expansion of the gap probability. It would
be interesting to see how the more precise asymptotics
compare to the $\beta$-ensemble case.

\begin{proof}[Proof of Theorem \ref{thmclt} (Central limit theorem)]
$ $

\noindent By (\ref{SDEphase}) we have $\of{\varphi^0(\tau),
\varphi^{\la}(\tau)-\la}= (\xi_0+\xi_1,\xi_0+\xi_2)$ where
\[
\xi_0=\cB(\tau), \quad \xi_1=\int_0^\tau \re\left[e^{-i
\varphi^0(t)} d\cW \right], \quad \xi_2=\int_0^\tau
\re\left[e^{-i \varphi^\la(t)} d\cW \right].
\]
Clearly, $\xi_0, \xi_1$ and $\xi_2$ are Gaussians with the
appropriate means and   variances and $\xi_0$ is
independent of $(\xi_1, \xi_2)$. However, the joint
distribution of $(\xi_1, \xi_2)$ is not Gaussian.  So we
need to prove that as $\lambda\to\infty$ the joint weak
limit of $(\xi_1,\xi_2)$ exists and it is given by a pair
of independent normals. Let $\cZ(t)=\int_0^t e^{-i
\varphi^0(s)} d\cW$ and $\cZ=\frac1{\sqrt{2}}(\cB_1+i
\cB_2)$ then
\[
\xi_1=\frac1{\sqrt{2}} \int_0^\tau d\cB_1, \qquad
\xi_2=\frac1{\sqrt{2}} \int_0^\tau \cos(\alpha^\la)
d\cB_1+\frac1{\sqrt{2}} \int_0^\tau \sin(\alpha^\la)
d\cB_2:=\xi_{2,1}+\xi_{2,2}.
\]
We will show that $(\xi_1, \xi_{2,1},\xi_{2,2})$ converges
weakly to three independent mean zero normals with
variances $\tau/2, \tau/4, \tau/4$. It is enough to prove
that for any $(a_1,a_2,a_3)\in \Real^3$ the random variable
$v^\la=a_1 \xi_1+a_2 \xi_{2,1}+a_3 \xi_{2,2}$ converges to
a mean zero normal with variance
$\tau(a_1^2/2+a_2^2/4+a_3^2/4)$. By representing the
Brownian integral as a time changed Brownian motion we can
see that $v^\la$  has the same distribution as $\hat
\cB\of{\frac12 \int_0^\tau (a_1+a_2
\cos(\alpha^\la))^2+a_3^2 \sin(\alpha^\la)^2  dt}$ for some
standard Brownian motion $\hat \cB$. All we need to show is
that
\[
2 \int_0^\tau (a_1+a_2 \cos(\alpha^\la))^2+a_3^2
\sin(\alpha^\la)^2 dt\to \tau(2a_1^2+a_2^2+a_3^2) \quad
\textup{in probability}.
\]
Using $\cos(x)^2=(\cos(2x)+1)/2$ and
$\sin(x)^2=(1-\cos(2x))/2$ this reduces to
\[
\int_0^\tau \cos(2\alpha^{\lambda}) dt=\re\int_0^1 e^{2i
\alpha^{\lambda}}  dt\to 0, \qquad \int_0^\tau
\cos(\alpha^{\lambda}) dt=\re\int_0^\tau e^{i
\alpha^{\lambda}} dt\to 0
\]
in probability. We work out the second claim, as the first
one can be done the same way. Using (\ref{sineSDE}) and
It\^o's formula we get
\[
\frac{1}{i \lambda} d \of{ e^{i \alpha^{\lambda}}} =e^{i
\alpha_\lambda}dt+\frac{\sqrt{2} }{\lambda}
\sin(\alpha^{\lambda}/2) d\cB^\lambda
+\frac{i}{\lambda}\sin(\alpha^{\lambda}/2)^2 dt
\]
and
\[
\int_0^\tau e^{i \alpha^{\lambda}}
dt=\frac{1}{i\lambda}\of{e^{i
\alpha_\lambda(\tau)}-1}-\frac{\sqrt{2}}{\lambda}
\int_0^\tau \sin(\alpha^{\lambda}/2) d\cB^\lambda+\frac{i
}{\lambda}\int_0^\tau \sin(\alpha^{\lambda}/2)^2 dt.
\]
As $\la\to \infty$ the first and third terms converge to 0
a.s., while the second term converges to 0 in $L^2$. This
means that their sum will converge to 0 in probability
which is what we needed to prove the joint limit theorem
for $(\varphi^{0}(\tau), \varphi^\lambda(\tau)-\lambda)$.

The limit  (\ref{cltformula}) follows from
\[
\sch[0,2\pi k+\theta]-k=\# \{ [\varphi^{0}(\tau),
\varphi^{\lambda/\tau}(\tau)]\cap 2\pi \Z \}-k=\left
\lfloor
\frac{\varphi^{\lambda/\tau}(\tau)-(2k\pi+\theta)+\theta}{2\pi}
\right \rfloor-\left \lfloor \frac{\varphi^0(\tau)}{2\pi}
\right \rfloor,
\]
where we also used that $\varphi^\lambda(t)$ is increasing in $\lambda$.
\end{proof}
The proof of Theorem \ref{cltsineb} is very similar.
\begin{proof}[Proof \ref{cltsineb} (Central limit theorem for $\textup{Sine}_\beta$)]
We will consider (\ref{sineb}) recalling that $\cZ$ is a complex Brownian motion with independent standard real and imaginary parts and hence
\begin{equation}\label{sineb1}
d \alpha^{\lambda}=\lambda \frac{\beta}{4}
e^{-\frac{\beta}{4}t} dt+ 2 \sin(\alpha^{\lambda}/2) d\cB,
\quad \alpha^\lambda(0)=0\quad t\in[0,\infty).
\end{equation}
First note that $\tilde \alpha(t)=\alpha^{\lambda}(T+t)$
with $T=\frac{4}{\beta} \log(\beta \lambda/4)$ satisfies
the same SDE with $\lambda=1$. Therefore
$$\frac{\alpha^{\lambda}(\infty)-\alpha^\la(T)}{\sqrt{\log(\lambda)}}\to 0$$
in probability. So it suffices to find the the weak limit
of $$\frac{\alpha^{\lambda}(T)-\lambda}{2\pi \sqrt{\log
\lambda}}.$$

We have  $$\alpha(T)-\lambda=-\frac{4}{\beta}+\int_0^T
2\sin(\alpha^{\lambda}/2) d\cB$$
which means
$$\alpha(T)-\lambda+\frac{4}{\beta}\eqd \hat \cB\of{\int_0^T
4\sin(\alpha^{\lambda}/2)^2 dt}$$ for a certain standard
Brownian motion $\hat \cB$. In order to prove the required
limit in distribution
 we only need to show that $\frac{4}{\log \lambda}\int_0^T  \sin(\alpha^\la/2)^2 dt\to \frac{8}{\beta}$ in probability. We have
\[
\frac{4}{\log \lambda}\int_0^T  \sin(\alpha^\la/2)^2 dt=\frac{8 \log\left[\beta  \lambda /4\right]}{\beta \log \lambda}+\frac{2}{\beta \log \lambda} \int_0^T \cos(\alpha^\la) dt.
\]
The first term converges to $8/\beta$.  To bound the second term we compute
\begin{eqnarray*}
\frac{4}{i \beta \lambda \log \lambda} d\of{e^{i \alpha^\la+\beta t/4}}&=&\frac{e^{i \alpha^\la}}{\log \la} dt+\frac{8}{\beta \lambda \log \lambda} e^{i \alpha^\la+\beta t/4} \sin(\alpha^\la/2) d\cB\\
&&+\frac{8i}{\beta \lambda \log \lambda} e^{i \alpha^\la+\beta t/4} \sin(\alpha^\la/2) ^2 dt+\frac{1}{i \lambda \log \lambda} e^{i \alpha^\la+\beta t/4}dt.
\end{eqnarray*}
The integral of the left hand side is $
\frac{4}{i \beta \lambda \log \lambda}\left[4 e^{i \alpha^\la(T)} \lambda /\beta- 1 \right]=O((\log \la)^{-1})
$. The integrals of the last two terms in the right hand side are of the order of $(\lambda \log \lambda)^{-1} \int_0^T e^{\beta t/4} dt=O((\log \la)^{-1})$. Finally, the integral of the second term on the right has an $L^2$ norm which is bounded by $C (\log \la)^{-1}$. This means the integral of the first term on the right, $(\log \la)^{-1} \int_0^T e^{i \alpha^\la} dt$ converges to 0 in probability from which the statement of the theorem follows.
\end{proof}

\section{The Brownian carousel}\label{s:carousel}

The SDE system (\ref{SDErelphase}) has a geometric
interpretation using the Brownian carousel introduced in
\cite{VV}. Recall the SDE system \eqref{SDErelphase}
\begin{equation}\label{SDErelphase1}
d\alpha^\la(t)=\la dt+ \re\left[
\of{e^{-i\alpha^\la(t)}-1}d\cZ \right],\quad
\alpha^\la(0)=0,
\end{equation}
Here $ \cZ$ is a standard complex Brownian
motion.

Consider the hyperbolic plane, let $x_0$ be a point on the
boundary and let $\mathcal{V}(t), t\ge 0$ be hyperbolic
Brownian motion. For a given $\lambda\in \Real$  we rotate
the boundary point $x_0$ about the moving center
$\mathcal{V}( t)$ with a constant angular speed $ \lambda$
and denote its position by $x^\la(t)$. This is the
\emph{Brownian carousel} with constant speed function $1$
and in Section 2 of \cite{VV} it was proved that the
hyperbolic angle determined by the points $x_0,
\mathcal{V}( t), x^{\la}(t)$  satisfies the SDE
(\ref{SDErelphase1}).

The evolution of $x^\la(t)$ can be described by an ODE.
Consider the Poincar\'e disk model for the hyperbolic
plane. Then  the boundary points are points on the unit
circle which can be described by an angle $x^\la(t)=e^{i
\gamma^\la(t)}$.  The hyperbolic Brownian motion
$\mathcal{V}$ in this model satisfies the following SDE:
\begin{equation}\label{eq:hypBM}
d\mathcal V= \frac{1-|\mathcal{V}|^2}{2} d\cY
\end{equation}
where $\cY$ is a standard complex Brownian motion. If we
set $x_0=1, \gamma^\la(0)=0$ then
\begin{equation}
\partial_t \gamma^\la=\la \frac{ |e^{i \gamma^\la}-\mathcal{V}|^2}{1-|\mathcal{V}|^2}, \quad \gamma^\la(0)=0.\label{odegamma}
\end{equation}
It is clear that this ODE system has a unique solution
which is analytic and strictly increasing in $\la$ for any
$t>0$. Note that one usually cannot get $\alpha^\la(t)$
from $\gamma^\la(t)$, however $\alpha^\la(t)\in 2\pi \Z$
if and only if $\gamma^\la(t)\in 2\pi \Z$.

Next we will prove Theorem \ref{thm:carousel}: if we add a
random shift to $\sch$ then the resulting point process can
be described with a Brownian carousel.

\begin{proof}[Proof of Theorem \ref{thm:carousel}]
Let $U$ be uniform on $[0,2\pi]$ and independent of $\sch$.
By Corollary \ref{limitPointProcess} and Remark \ref{rem5}
the point process $\sch+U$ has the same distribution as the
solutions of the equation
$\varphi^{\lambda/\tau}(\tau)=U\mod 2\pi$. This can be
rewritten as
\[
\alpha^\la(\tau)=-\varphi^0(\tau)+U \mod 2\pi.
\]
Since $U$ is independent of $\alpha^\la$ and $\varphi^0$,
we have
$$(\alpha^\la(\tau),-\varphi^0(\tau)+U) \mod 2\pi \quad \eqd \quad (\alpha^\la(\tau),U) \mod 2\pi.$$ Thus we can
just look at the solutions of $\alpha^\lambda(\tau)=U \mod
2\pi$.

For a given $u\in[0,2\pi)$ the solution set of
$\alpha^\la(\tau)=u \mod 2\pi$ can be described by the
carousel construction: it is given by the set of those
$\la\in\Real$ for which the hyperbolic angle $x_0,
\mathcal{V}(T), x^\la(T)$ is equal to $u$.

By the Markov property of hyperbolic Brownian motion, the
hyperbolic angle $x_0, \mathcal{V}(T), \mathcal
{V}(\infty)$ is just uniform and independent of $V$ on the
time interval $[0,T]$, so we may as well call it $U$. The
claim follows.
\end{proof}

We also provide an alternate proof to a version of Theorem
\ref{repulsion} using the Brownian carousel.
\begin{thm}[Eigenvalue repulsion]\label{repulsion2} For $\mu\in \Real$ and $\eps>0$
we have
\begin{equation}\label{bound2eig2}
\Prob\set{\sch[\mu,\mu+\eps]\geq 2}\leq
4\exp\left(-\frac{(\log(2\pi/\eps)-\tau-1)^2}{\tau}\right).
\end{equation}
whenever the squared expression is nonnegative.
\end{thm}

\begin{proof}
As in the first proof we can assume that $\mu=0$. By  (\ref{e:gapbound}) if there are at least  two points in $[0,\eps]$ then the relative phase function must be at least $2\pi$ which means that the Brownian carousel had to take at least one full turn. Thus
\[
\Prob\set{\sch[0,\eps]\geq 2}\leq \Prob
\set{\alpha^{\eps/\tau}(\tau)\geq 2\pi}=\Prob \set{\gamma^{
\eps/\tau}\of{\tau}\geq 2\pi}.
\]
where $\gamma$ is the solution of \eqref{odegamma}. From
(\ref{odegamma}) we get
\[
\gamma^{\eps/\tau}(\tau)\le  \max_{0\le t \le \tau}
(1-|\cV_t|^2)^{-1}=(1-\max_{0\le t \le \tau}|\cV_t|^2)^{-1}
\]
which means that
\begin{equation}
\gamma^{\eps/\tau}\of{\tau}\geq 2\pi \quad\Rightarrow\quad
1-\frac{\eps}{2 \pi}\le \max_{0\le t\le \tau}
|\mathcal{V}_t|^2.\label{aux1}
\end{equation}
In the Poincar\'e disk model the hyperbolic distance between the origin and a point $z$ in the unit disk is given by $q(z)=\log\of{\frac{1+|z|}{1-|z|}}$. Thus (\ref{aux1}) implies
\[
\max_{0\le t\le\tau} q(\mathcal{V}_t)\ge \log\of{2\pi/
\eps}.
\]
The probability that the hyperbolic Brownian motion leaves
a ball  with a large radius $r$ in a fixed time is
comparable to the probability that a one-dimensional
Brownian motion  leaves $[-r,r]$ in the same time. This
follows  by noting that It\^o's formula with
(\ref{eq:hypBM}) gives

\[
dq =\frac{dB}{\sqrt{2}}+ \frac{\coth(q)}{4} dt
\]
for the evolution of $q(\cV)$ with a standard Brownian
motion $B$. By increasing the drift from $\coth(q)/4$ to
$\infty {\bf 1}_{q\in [0,1]}+\coth(1)/4$ we see that  $q$
is stochastically dominated by $1+t\coth(1)/4 +
|B(t)|/\sqrt{2}$ where $B$ is standard Brownian motion and
$\coth(1)<4$. Thus
\begin{eqnarray*}
\Prob\of{\max_{0\le t\le\tau } q(\mathcal{V}_t)\ge
\log\of{2\pi/ \eps}}&\le&
\Prob\of{\max_{0\le t\le\tau} |B(t)|\ge \log\of{2\pi/\eps}-1-\tau}\\
&\le&
4\exp\left(-\frac{(\log(2\pi/\eps)-\tau-1)^2}{\tau}\right)
\end{eqnarray*}
which proves the theorem.
\end{proof}

\section{Convergence of the regularized transfer matrix evolution}\label{s:SDEconvergence}
\newcommand{\zx}{X}
\newcommand{\dzx}{{\mathcal{Y}}}
This section is devoted to the proof of Theorem \ref{DiffusionTransfer}. In order to keep the notation simple, we will only treat the case when $m=1$, i.e.~when a single value $\la\in\Comp$ is fixed. The extension to $\ula=(\la_1,\cdots,\la_m)$ for $m>1$ is straightforward. We drop $\la$ from the notation.
The identity
\[T(y)T^{-1}(x)=I+\mat{0}{y-x}{0}{0}\]
 and the recursion $M_\ell =T(E+\eps_\ell )M_{\ell -1}$ from (\ref{Mn}) implies
\begin{eqnarray}\nonumber
\X_\ell 
&=&T^{-\ell }(E)\set{T(E+\eps_\ell )T^{-1}(E)}T^{\ell}(E)\X_{\ell -1}\\
&=&T^{-\ell }(E)\mat{1}{\eps_\ell}{0}{1}T^{\ell}(E)\X_{\ell
-1}. \label{recXn}
\end{eqnarray}
This shows that $\X_\ell , 0\leq \ell \leq n$ is a Markov
chain,   the initial term $\X_0=I$. As $E$ is fixed, from
now on we write $T$ for $T(E)$. We will work in the basis
diagonalizing $T$ i.e. we consider $\zx_\ell =Z^{-1}\X_\ell
Z $ instead of $\X_\ell$. (We have learned that such a
change of basis has been considered for a slightly
different problem by \cite{HSB}). Using $T=ZDZ^{-1}$ from
(\ref{zdef}), we obtain after simplification
$$T^{-\ell} \mat{0}{\eps_\ell}{0}{0} T^{\ell}=\frac{i\rho\eps_{\ell }}{2} Z
O_{\ell }Z^{-1},\qquad O_{\ell }:=\mat{1}{z^{2\ell}}{-\bar
z^{2\ell}}{-1}.$$ Therefore $\zx_\ell $ is a Markov chain
with the initial condition $\zx_0=I$ and given by the
recurrence
\begin{equation}\zx_{\ell }=\zx_{\ell -1} +U_\ell  \zx_{\ell -1},\qquad U_\ell=U_\ell^n :=i \rho \eps_{\ell }
O_{\ell }/2\label{Xdef}\end{equation} Because of the
oscillating factors $z^{\pm(2\ell)}$, the term $U_\ell
\zx_{\ell -1}$ is too rough to approximate a stochastic
differential. However, on a mesoscopic scale $1\ll K\ll n$,
the difference $X_{\ell+K}-X_{\ell}=\sum_{j=1}^{K}U_{\ell
+j}X_{\ell +j-1}$ becomes a good approximation for a
stochastic differential because the oscillations cancel in
the sum. For convenience of the reader, we first present a
heuristic derivation of the limiting SDE and then we give a
rigorous proof.
\begin{proof}[Heuristic proof
]
As $\zx_{\ell+K}=(I+U_{\ell
+K})\cdots(I+U_{\ell+2})(I+U_{\ell +1})\zx_{\ell}\simeq
(I+\sum_{j=1}^{K}U_{\ell+j})\zx_{\ell}$, we have
$$X_{\ell +K}-X_\ell \simeq \sum_{j=1}^{K}U_{\ell +j}X_{\ell }=\frac{i\rho}{2} \of{\sum_{j=1}^{K}\eps_{\ell +j}O_{\ell +j}}X_{\ell }.$$
We look separately at the drift and the noise contributions, i.e. we split
$$\frac{i \rho}{2} \sum_{j=1}^{K}\eps_{\ell +j}O_{\ell +j}=\of{\frac{i\la}{2n} \sum_{j=1}^{K}O_{\ell +j}} +\of{ - \frac{i\sigma\rho}{2 {n}^{1/2}} \sum_{j=1}^{K}\omega_{\ell +j}O_{\ell +j}}:= \cD + \cN.$$
Since $-2<E<2$, we have  $\abs{z}=1$ and $z^2\neq 1$, which implies that $\sum_{j=1}^K z^{\pm (2\ell+2j)}$ is bounded for large $K$. With $\Delta t=K/n$ we then have
$$\cD\simeq \frac{\la \Delta t}{2} \mat{i}{0}{0}{-i}.$$
For the noise term we write
\[
\cN=\frac{-i\sigma\rho}{ 2 K^{1/2} } \sqrt{\Delta
t}\sum_{j=1}^{K}\omega_{\ell +j}O_{\ell
+j}=\frac{\sigma\rho}{ 2} \sqrt{\Delta t} \mat{i \xi_{\ell
,K}}{\zeta_{\ell ,K}}{\bar \zeta_{\ell ,K}}{- i \xi_{\ell
,K}}
\]
with
\begin{equation}
 \xi_{\ell,K}=-K^{-1/2} \sum_{j=1}^K \omega_{\ell+j}, \qquad \zeta_{\ell ,K}=-i K^{-1/2} \sum_{j=1}^K \omega_{\ell+j}  z^{2\ell+2j}.\label{eq:normal}
\end{equation}
In the limit $K\To\infty$, $(\xi_{\ell ,K}, \re \zeta_{\ell,K}, \im \zeta_{\ell,K})$ is a mean zero Gaussian vector $(\xi_{\ell}, \re \zeta_{\ell}, \im \zeta_{\ell})$
whose distribution is determined by the covariance matrix. Computing the covariance matrix is equivalent to computing the limits of the expectations of $ \xi_{\ell,K}^2,  \xi_{\ell,K}, \zeta_{\ell,K}, \zeta_{\ell,K}^2$ and $|\zeta_{\ell,K}|^2 $ since $\Exp \re \zeta\im \zeta=\frac{1}{2}\im \Exp \zeta^2$,
$\Exp (\re \zeta)^2=\frac12\of{\Exp|\zeta|^2+\re\Exp \zeta^2}$, $\Exp (\im \zeta)^2=\frac12\of{\Exp|\zeta|^2 - \re\Exp \zeta^2}$, $\Exp \xi \re \zeta=\re \Exp \xi \zeta$ and $\Exp \xi \im \zeta=\im \Exp \xi \zeta$. Using (\ref{eq:normal}) we get
\begin{equation}\label{covariances}
\Exp \xi_\ell^2=\Exp |\zeta_\ell^2|=1, \quad \Exp \xi_\ell \zeta_\ell=
\lim_{K\To\infty}i K^{-1} \sum_{j=1}^K  z^{2\ell+2j}\quad \textrm{ and } \quad
\Exp \zeta_\ell^2=-\lim_{K\To\infty}K ^{-1}  \sum_{j=1}^K  z^{4\ell+4j}.
\end{equation}
The first sum in \eqref{covariances} converges to zero. The assumption $E\in(-2,2)\backslash\{0\}$ implies that $|z^4|=1$, $z^4\neq 1$ and therefore the second sum in \eqref{covariances} converges to zero as well. Thus asymptotically $\xi_{\ell,K}$ and $\zeta_{\ell,K}$ are independent standard real and complex normals. Collecting our estimates  we formally get the SDE
\begin{equation}\label{XSDE}
dX=\mat{i \la/2}{0}{0}{-i \la/2} Xdt +\frac{
\sigma\rho}{2} \mat{i d \cB}{d \cW}{d\overline{\cW}}{-
id\cB}X,\qquad X(0)=I.
\end{equation}
from which Theorem \ref{DiffusionTransfer} would follow after rescaling time and $\lambda$. In the case $E=0$, we get $\Exp \zeta_\ell^2=-1$, which implies that asymptotically $\xi_{\ell,K}$ and $\im \zeta_{\ell,K}$ are independent standard normals and $\re \zeta_{\ell,K}=0$. In this case we formally get the SDE
\begin{equation}\label{formalSDE_E0}
dX=\mat{i \la/2}{0}{0}{-i \la/2} Xdt +\frac{ \sigma
\rho}{2} \mat{i d \cB_1}{id \cB_2}{-id \cB_2}{-
id\cB_1}X,\qquad X(0)=I,
\end{equation}
where $\cB_1,\cB_2$ are independent standard Brownian motions.
\end{proof}

As we will show these  computations can be made rigorous.

\begin{proof}[Proof of Theorem \ref{DiffusionTransfer}]
In order to make the convergence argument precise, we use Proposition \ref{convergence} which is a slight modification of Proposition 23 in \cite{VV}. We show the convergence in case of a single $\lambda\in \Comp$, the proof for finite dimensional marginals in $\lambda$ is very similar.

We will prove  that $X^n_\ell=Z^{-1} Q_\ell^n Z$ converges to the solution of the SDE (\ref{XSDE}), from this the statement of the theorem follows. We can identify the $2\times 2$ complex matrix $X_\ell^n$ with a vector in $ \Real^8$ by taking the real and imaginary parts of the entries. From (\ref{Xdef}) one gets that the conditional distribution of $X^n_{\ell+1}-X^n_\ell$ given $X^n_\ell=x$ is the same as that of
\[
Y^n_\ell(x):= \left(\frac{i \la}{2 n}-\frac{i \sigma\rho
\omega_{\ell+1}}{2\sqrt{n}}\right)
\mat{1}{z^{2(\ell+1)}}{-\bar z^{2(\ell+1)}}{-1} x.
\]
From this $b^n(t,x)$ and $a^n(t,x)$ are computable. The
function $b^n(t,x)$ will be a vector in $\Real^8$
corresponding to the complex matrix
\[
\frac{i \la}{2
} \mat{1}{z^{2(\ell+1)}}{-\bar z^{2(\ell+1)}}{-1} x\qquad \textup{with} \qquad \ell=\lfloor nt \rfloor.
\]
The asymptotic variance $a^n(t,x)$ is a bit more cumbersome to write down, it is an $8\times 8$ matrix with entries which are linear combinations of terms of the form of  $A_{j} A_{k}$, $A_{j} \bar A_{k}$ and $\bar A_{j} \bar A_{k}$ with $j,k \in \{1,2\}^2$  where
\[
A=-\frac{i \sigma\rho}{2} \mat{1}{z^{2(\ell+1)}}{-\bar
z^{2(\ell+1)}}{-1} x\qquad \textup{with} \qquad
\ell=\lfloor nt \rfloor.
\]
Clearly the coordinates of $a^n(t,x)$ are bilinear
functions of $x$ and $\bar x$ with bounded coefficients
(depending on $\sigma\rho$ and various powers of $z, \bar
z$).

The functions $a(t,x), b(t,x)$ can be obtained from $a^n, b^n$ by writing zeros in place of the (non-trivial) powers of $z$ and $\bar z$, these are clearly  $C^2$ functions. Condition (\ref{e fia}) follows from the fact that $n^{-1} \sup_\ell \sum_{j=1}^\ell z^{2j}$ and $n^{-1} \sup_\ell \sum_{j=1}^\ell z^{4j}$ both converge to 0. Because of this in the integrals of (\ref{e fia}) the $z$ terms will vanish in the limit and by the construction of $a$ and $b$ the other terms will  cancel. Condition (\ref{e lip}) is straightforward since $b, b^n$ are linear and $a, a^n$ are bilinear functions of $x, \bar x$ with bounded coefficients.  The condition (\ref{e 3m}) is a consequence of the assumption $\Exp|\omega_\ell|^3<\infty$ and since $X^n_0=I$ the last condition is also satisfied.

Thus we can apply Proposition \ref{convergence} and the only thing left is to show that the functions $a(t,x)$, $b(t,x)$ correspond to the variance and drift functions corresponding to (\ref{XSDE}). The fact that the drift function agrees is straightforward. To check the variance one needs to turn (\ref{XSDE}) into a real vector valued SDE which basically means that we need to take independent standard real and complex standard normals $B$ and $W$ and compute the variance of the random vector corresponding to
\[
-\frac{ \sigma\rho}{2} \mat{i B}{ W}{\bar W}{-i B} x
\]
Using $\Exp B^2=\Exp |W|^2=1$ and $\Exp BW=\Exp W^2=0$ one
can check that we get exactly $a(t,x)$ which finishes the
proof of \eqref{XSDE}. A time-change and the
reparametrization $\lambda\to \lambda/\tau$ gives the SDE
that is independent of $\sigma\rho$
\begin{equation}\label{XSDEnewtime}
dX=\frac{1}{2}\mat{i \la}{0}{0}{-i \la} Xdt +\frac{ 1}{2}
\mat{i d \cB}{d \cW}{d\overline{\cW}}{- id\cB}X,\qquad
X(0)=I.
\end{equation}
and we get the claimed SDE through multiplication on the
left by the matrix $Z$.
\end{proof}

The same argument works for  the proof of the first part of Theorem \ref{theoremE=0}. The only difference is that in that case $z=i$ thus $z^{4j}=\bar z^{4j}=1$ and $a(t,x)$ will be defined accordingly.

\section{Convergence of the rescaled eigenvalue process}\label{s:pointprocess}

In this section we prove the delocalization result (Theorem \ref{delocalization}) and the point process limit theorems (Corollary \ref{ConvergenceOfPoint} and \ref{limitPointProcess}).

\subsection*{Tightness bounds}
\begin{lem}\label{analyticmartingales} Let $(X_k(z): 1\leq k\leq n,z\in\Comp)$ be random $d_1\times d_2$ matrices whose entries have finite second moments. Assume
that $X_k(z)$ is analytic in $z$ and it is a martingale
with respect to a filtration $\cF_k$. Then for every
$r_1<r_2<\infty$ and $t>0$
$$\Prob\set{\max_{1\leq k\leq n, \abs{z}\leq r_1}\Tr X_k(z)X_k(z)^*\geq t}\leq t^{-1}\frac{r_2+r_1}{r_2-r_1} \frac{1}{2\pi}\int_{-\pi}^{\pi} \Exp\Tr X_n(e^{i\theta r_2}) X_n(e^{i\theta r_2})^* d\theta $$
\end{lem}
\begin{proof} Since each entry $X_k(z)(i,j)$ is analytic in $z$, the Poisson formula and Jensen's inequality
gives $$\abs{X_k(z)(i,j)}^2\leq \frac{r_2+r_1}{r_2-r_1}
\frac{1}{2\pi} \int_{-\pi}^{\pi} \abs{X_k(e^{i\theta
r_2})(i,j)}^2 d\theta, \qquad \mbox{for} \abs{z}\leq r_1.$$
Summing over all $i,j$ gives
$$\Tr X_k(z)X_k(z)^*\leq \frac{r_2+r_1}{r_2-r_1} \frac{1}{2\pi}
x_k, \qquad \mbox{for} \abs{z}\leq r_1,$$ where
$$x_k:=\int_{-\pi}^{\pi} \Tr X_k(e^{i\theta r_2}) X_k(e^{i\theta r_2})^* d\theta$$
As $x_k$ is a submartingale, the statement follows from
Doob's inequality and Fubini's theorem.
\end{proof}
\begin{proof}[Proof of Theorem \ref{delocalization}] Let $\Xi_\lambda=\Xi=T(E+\frac{\la}{\rho n})$. For large enough $n$ and all complex $\la$, $\abs{\la}\leq r_2:=2R$,
the eigenvalues and eigenvectors of $\Xi$ are close to those of $T(E)$, we can write
\begin{equation}\label{xidecomp}
\Xi_\lambda=Z_\lambda D_\lambda Z_\lambda^{-1}
\end{equation}
with $\norm{Z_\lambda-Z}_2, \norm{D_\lambda-D}_2, \norm{Z_\lambda^{-1}-Z^{-1}}_2$ all bounded by $c/n$
 with $c$ depending on $E, R$.
 Since $T(E)$ has unit length eigenvalues we can find another constant $C=C(R,E)$ so that the eigenvalues of $\Xi^k, |k|\le n$ are uniformly bounded by $C$. Using this with the decomposition  (\ref{xidecomp}) we get that $\Tr \Xi^k \Xi^{*k}$ is also uniformly bounded for $|k|\le n$.

Setting
$S_\ell =\Xi^{-\ell }M_\ell $ we have, analogously to
\eqref{recXn}, that
\begin{equation}
S_\ell =S_{\ell -1}+ \cE_\ell  S_{\ell -1}, \qquad \cE_\ell  := \Xi^{-\ell }
\mat{0}{-\sigma n^{-1/2}\omega_\ell}{0}{0}
\Xi^{\ell}.\label{Srec}
\end{equation}
Since $\ev \cE_\ell=0$ we get
$$\Exp S_\ell S_\ell ^*=\Exp S_{\ell -1}S_{\ell -1}^*+\Exp \cE_\ell S_{\ell -1}S_{\ell -1}^*\cE_\ell ^*.$$
Taking the trace and conditioning on $S_{\ell -1}$ we get
\begin{align*}
\Tr \Exp S_\ell S_\ell ^* -\Tr \Exp S_{\ell -1}S_{\ell
-1}^*&=\Exp\Tr \cE_\ell S_{\ell -1}S_{\ell -1}^*\cE_\ell
^* =\Exp\Tr S_{\ell -1}S_{\ell -1}^*\cE_\ell
^*\cE_\ell\\ &=\Tr (\Exp S_{\ell -1}S_{\ell -1}^*)(\Exp
\cE_\ell ^*\cE_\ell )\leq \Tr (\Exp S_{\ell -1}S_{\ell
-1}^*)\Tr (\Exp \cE_\ell ^*\cE_\ell)
\end{align*}
In the last step we used that if $A, B$ are positive semidefinite matrices of the same dimension then $ \Tr AB\le \Tr A \Tr B$.
Using the (\ref{xidecomp}) and the bounds on $Z_\lambda, D_\lambda$ one gets that  $\Tr (\Exp \cE_\ell ^*\cE_\ell)\le c \sigma^2/n$ and it follows that
  for all  $n\geq n_0$, $0\leq \ell\leq n$  and $\la\in\Comp$ with $\abs{\la}\leq r_2$ we have
 $$\Exp[\Tr(S_\ell ^*S_\ell)]\leq C_1.$$
Note that $(S_\ell : 0\leq \ell\leq n)$ is martingale
analytic in the parameter $\la$. Then Lemma
\ref{analyticmartingales} implies that \eqref{boundSn}
holds for $S_\ell ^\la$ in instead of $M_\ell ^\la$ with probability $1-c/t$. To translate the
result  for $M_\ell ^\la$ we use
 the estimate
$$\Tr M_\ell M_\ell^*=\Tr \Xi^{*\ell}\Xi^{\ell}S_\ell S_\ell^*\leq  \Tr \Xi^{*\ell}\Xi^{\ell} \, \Tr S_\ell S_\ell^*\le C  \Tr S_\ell S_\ell^*.$$


To prove the second part of the theorem it is enough to show that if we assume
that $\Tr M_\ell M_\ell^*$ is bounded by $t$ uniformly in $\ell$ and $\lambda$ then (\ref{boundpsi}) holds.
 Let
$\psi$ be a normalized eigenvector of $H_n$ corresponding
to the eigenvalue $e=E+\frac{\la}{\rho n}\in
[E-\frac{R}{\rho n},E+\frac{R}{\rho n}]$ and let $\Psi_\ell
= \binom{\psi_{\ell+1}}{\psi_{\ell}}$. For  each $0\leq
k,\ell \leq n$, the transfer matrix description of the
eigenvalue equation gives $\Psi_k=M_k \Psi_0=M_k
(M_\ell)^{-1} \Psi_\ell$, $\psi_0=\psi_{n+1}=0$.
Since for the induced operator norm $\norm{\cdot}_{2,2}$ we have
$$\norm{A}_{2,2}=\sqrt{\lambda_{\max}(AA^*)}\le \sqrt{\Tr A A^*},$$
 we get the bound $\norm{M_\ell}_{2,2}< \sqrt{t}$.
But $M_\ell$ is a $2\times2$ matrix and  $\det M_\ell=1$ so $\norm{M_\ell}_{2,2}=\norm{M_\ell^{-1}}_{2,2}$. This leads to   $\norm{M_\ell^{-1}}_{2,2}< \sqrt{t}$ and
 $$\norm{\Psi_k}_2^2 \le \norm{M_k}_{2,2}^2\norm{M_\ell^{-1}}_{2,2}^2 \norm{\Psi_\ell}_2^2<t^2 \norm{\Psi_\ell}_2^2.$$
Summing
the last inequality over all $0\leq k \leq n$ gives $2 <
(n+1)t^2 \norm{\Psi_\ell}^2$ and summing over all $0\leq
\ell \leq n$ gives  $(n+1)\norm{\Psi_k}^2 < 2 t^2$.
\end{proof}

\subsection*{Proof of Corollary \ref{ConvergenceOfPoint} }
By (\ref{ev_cond}) and (\ref{Xn}) for each $n$, the
rescaled eigenvalues $\la_k$ are given by the zeros of the
random analytic function $g_n:\Comp\To\Comp$:
\begin{equation}\label{defgn}
g_n(\la):=\det \of{ Q_n^\la\bin{1}{0} , B_n }, \qquad
B_n:=T^{-n}\bin{0}{1}.
\end{equation}
Our assumption is that along a subsequence $n_j$, $B_n$
converges to a vector $B\in\Real^2$. It follows from
Theorem \ref{DiffusionTransfer} that for any fixed
$(\la_1,\cdots,\la_m)\in\Comp^m$, the random vector
$(g_{n_j}(\la_1),\cdots,g_{n_j}(\la_m))\in\Comp^m$
converges in distribution to a random vector
\begin{equation}\label{defg}
(g(\la_1),\cdots,g(\la_m)):=\of{\det \of{
Q^{\la_1/\tau}(\tau)\bin{1}{0} , B },\cdots,\det \of{
Q^{\la_m/\tau}(\tau)\bin{1}{0} , B }}
\end{equation}
where $(Q^{\la_1/\tau}(t),\cdots Q^{\la_m/\tau}(t))$ is the solution
to the SDE \eqref{LimitingTransfer}. We need to show that
the family of distributions in \eqref{defg}, indexed by
$(\la_1,\cdots,\la_m)\in \Comp^m$, defines a random
analytic function $g(\la)$ and that the mode of convergence
$g_{n_j}(\la)\To g(\la)$ is strong enough to ensure
convergence of zeros.

We will use the following notions of convergence. Let $\Fs$
denote the space of analytic functions from a connected
open set $D$ in $\Comp$ to $\Comp^d$. We equip $\Fs$ with
the metric
\[
d(f,g):=\sum_{r=1}^{\infty}
2^{-r}\frac{\norm{f-g}_r}{1+\norm{f-g}_r}, \quad
\textup{where}\quad\norm{h}_r:=\max_{z\in D\cap
\{|z|<r\}}\norm{h(z)}.
\]
Then $(\Fs,d)$ is a complete separable metric space and
convergence in $d$ is the local uniform convergence. A
random analytic function in $\Fs$ is a measurable mapping
$\omega\To f^\omega$ from a probability space
$(\Omega,\cF,P)$ to $(\Fs,\cB)$, where $\cB$ is the Borel
$\sigma$-field generated by the metric $d$. The law of $f$
is the induced probability measure $\rho_f$ on
$(\Fs,\cB_d)$. A sequence $f_\ell ^\omega$ of random
analytic functions is said to converge in law to a random
analytic function $f^\omega$ if $\rho_{f_\ell }\To \rho_f $
in the usual sense of weak convergence.
\begin{prop}\label{convergence_analytic_functions} Suppose
\newline\noindent{\rm(1)}
$f_\ell ^\omega$ is a sequence of random analytic functions
in $\Fs$ such that for every $0<r<\infty$,
\begin{equation}\label{general_tightness}
\lim_{C\To\infty}\Prob\set{\sup_{\ell \geq 1,\abs{\la}\leq
r} \abs{f_\ell (\la)}>C}=0.
\end{equation}
\newline\noindent{\rm(2)} for each $m\geq 1$ and $\tla=(\la_1,\la_2,\cdots,\la_m)\in \Comp^m$ there is a probability distribution $\nu^{\tla}$ on $\Comp^{m}$ and the random vector
$(f_\ell ^\omega(\la_1),f_\ell ^\omega(\la_2),\cdots,f_\ell
^\omega(\la_m))\in\Comp^m$ converges in law to
$\nu^{\tla}$.

Then there is a random analytic function $f^\omega$ in
$\Fs$  such that $f_\ell ^\omega$ converges in law to
$f^\omega$. Moreover for each
$\tla=(\la_1,\la_2,\cdots,\la_m)\in \Comp^m$,
$(f^\omega(\la_1),f^\omega(\la_2),\cdots,f^\omega(\la_m))\in\Comp^m$
has distribution $\nu^\tla$.
\end{prop}
\begin{proof} For each disk $D_r:=\set{\la\in\Comp:\abs{\la}<r}$, the bound in \eqref{general_tightness} together with Montel's and Prokhorov's theorems imply that a subsequence of $f_\ell $ restricted to $D_r$ converges in law to a random analytic function $f_r$ on $D_r$. Then by a diagonal argument, there is a subsequence of $f_{\ell _k}$ such that for each integer $r$, the restriction of $f_{\ell _k}$ to $D_r$ converges to to random analytic function $f_r$ on $D_r$. The distributions of the functions $f_r$ are consistent with respect to restricting to smaller discs, and thus there is a random analytic function $f$ on $\Comp$ such that $f_{\ell _k}\To f$ in law. Condition $(2)$ is strong enough to ensure that $f$ is unique and thus $f_\ell \To f$.
\end{proof}

Let $\cA=\cA(\Comp,\Comp)$ and
$\cA_0:=\cA\backslash\set{0}$, i.e. we discard the
identically zero function. $\cM$ denotes the set of
nonnegative Borel measures on $\Comp$, that are finite on
bounded subsets of $\Comp$. We consider the local weak
topology on $\cM$: a sequence $\mu_\ell \in\cM$ is said to
converge to $\mu\in\Comp$ if for every continuous function
$\psi:\Comp\To\Real$ of compact support, $\int \psi
d\mu_\ell \To\int \psi d\mu$. For $f\in \cA_0$, we denote
by $\mu_f$ the zero counting measure of $f$, i.e.
$\mu_f=\sum_{f(z)=0}m(z)\delta(z)$, where $m(z)$ is the
multiplicity of the zero $z$. As an elementary consequence
of Cauchy's integral formula, we have that for $f_\ell
,f\in \cA_0$, $d(f_\ell ,f)\To 0$ implies $\mu_{f_\ell }\To
\mu_f$.

A random measure in $\cM$ is a measurable function
$\omega\To\mu^\omega$ to $\cM$ (with the Borel
$\sigma$-algebra).

If $f^\omega$ is a random analytic function in $\cA$ with
$\Prob(f\equiv 0)=0$, then $\mu_{f^\omega}$ is a random
measure in $\cM$. If $f_\ell ^\omega$ converges in law to
$f^\omega$ and $\Prob(f_\ell \equiv0)=\Prob(f \equiv 0)=0$, then the corresponding random
measure $\mu_{f_\ell ^\omega}$ converges in law to
$\mu_{f^\omega}$.

We can now complete the proof of Corollary
\ref{ConvergenceOfPoint}. The appropriate part of Theorem
\ref{theoremE=0} can be proved the same way.
\begin{proof}[Proof of Corollary \ref{ConvergenceOfPoint}]
The tightness \eugene{bound \eqref{boundSn}}, Theorem
\ref{DiffusionTransfer} and Proposition
\ref{convergence_analytic_functions}  guarantee that the
random analytic function $\Comp\ni\la\To Q^{\la}_n\in
M_2(\Comp)$ converges in distribution to the random
analytic function $\Comp\ni\la\To Q^{\la/\tau}(\tau)\in M_2(\Comp)$
as $n\To\infty$. Then the random analytic function
$g_{n_j}(\la)$ defined in \eqref{defgn} converges to the
random analytic function $g(\la)$ defined in \eqref{defg}.
In addition it is easy to see that $\Prob(g_n \equiv0)=\Prob(g \equiv 0)=0$. Thus $\mu_{g_{n_j}^\omega}$
converges in law to $\mu_{g^\omega}$.
\end{proof}

\subsection*{The phase function}

\begin{proof}[Proof of Corollary \ref{limitPointProcess}]
The existence and uniqueness of the analytic solution of
(\ref{SDEphase}), as well as the monotonicity of
$\varphi^\la(\tau)$ will be shown in Section \ref{s:sde}.

\newcommand{\et}{e^{i\theta}}

To prove the second part of the theorem we will first
assume that $z^{n_j+1}$ converges to $e^{i \theta}$. Then
$T^{n_j}(E)$ converges to a matrix $\tilde T$ and
\[
\tilde T Z=\lim T^{n_j}(E)  Z=Z\lim \mat{\bar
z^{n_j}}{0}{0}{z^{n_j}}
 = \mat{\bar z}{z}{1}{1}
  \mat{z e^{-i \theta}}{0}{0}{\bar z e^{i\theta}}
 = \mat{e^{-i
\theta}}{e^{i \theta}}{z e^{-i \theta}}{\bar z e^{i\theta}}.
\]
By Corollary \ref{ConvergenceOfPoint} we need to identify
the zeros of $$[\tilde T Q^\lambda(\tau)]_{11}=[\tilde T Z
\tilde X^\lambda(\tau) Z^{-1}]_{11}=[\tilde T Z
X^\lambda(\tau) ]_{11}$$ where $\tilde X^\lambda$ satisfies
the SDE (\ref{XSDEnewtime}). By linearity, $X:=\tilde
XZ^{-1}$ satisfies the same SDE, but with initial condition
$X^\la(0)=Z^{-1}$.

Note that if $\lambda\in \Real$ then $X^{\la}$ is a matrix
of the form $\mat{a}{b}{\bar a}{\bar b}$. Indeed, this
holds for $t=0$ and it is preserved by the evolution by
\eqref{XSDE}. So we have
\begin{equation}\label{11term}
[\tilde T Z X^{\la}(\tau)]_{11}=e^{-i \theta}
X^{\la}(\tau)_{11}+e^{i \theta}  X^{\la}(\tau)_{21}=2 \re
\left[e^{-i \theta}  X^{\la}(\tau)_{11}   \right].
\end{equation}
We rewrite the SDE's for the matrix entries as follows:
\begin{align}\label{SDEa}
2dX_{11}&=i \lambda X_{11} dt+i X_{11} d\cB+\bar X_{11} d\cW, \quad X_{11}(0)=i \rho/2,\\
2dX_{12}&=i \lambda X_{12} dt+i X_{12} d\cB+ \bar X_{12}
d\cW, \quad X_{12}(0)=i z \rho/2.\label{SDEb}
\end{align}
It\^o's formula gives
\[
d \det \tilde X^\la=0\qquad \textup{and} \qquad 2i
\,\im\left[X_{11} \bar X_{12}\right]=\det{\tilde
X^\la(t)}=\det Z^{-1}=i \rho/2
\]
which shows that $X_{11}^\la(t)$ is never equal to 0,  and
the phase function $\varphi^\la(t)$ is well-defined via
\begin{equation}\nonumber
e^{i\varphi^\la(t)}= \frac{i X_{11}^\lambda(t)}{\overline{i
X_{11}^\lambda(t)}}, \qquad \varphi^\lambda(0)=0.
\end{equation}
 It\^o's
formula applied to \eqref{SDEa} shows that $\varphi$
satisfies \eqref{SDEphase} with $- d\cW$ in place of
$d\cW$. Also, the zeros of \eqref{11term} are given by the
solutions of $\re[e^{-i \theta-i\pi/2}
iX^{\la}(\tau)_{11}]=0$, or, equivalently, $- \theta-\pi/2+
\varphi^\la(\tau)/2 \in\pi
\Z $. 

So far we have shown that if $z^{n_j+1}\to e^{i \theta}$
then
$$\Lambda_{n_j}\Rightarrow \set{\lambda:\varphi^{\la/\tau}(\tau)\in 2\theta+\pi+2\pi \Z}\eqd\sch+2\theta+\pi$$
where the last equality follows by the definition
\eqref{e:schdef} and Lemma \ref{invariance}. It follows that
\[\Lambda_n-2\arg(z^{n+1})-\pi\Rightarrow \sch.\]
Since $2\arg(z^{n+1})-\arg(z^{2n+2})$ is either 0 or $2\pi$ and $\sch\eqd \sch+2\pi$ the statement of the corollary follows.
\end{proof}

\section{The limit theorem for the the decaying model}\label{s:decaying}

In this section we discuss Theorems \ref{thmdecay1} and \ref{limitPointProcessDecayingCase}. The proof of Theorem \ref{thmdecay1} can be done exactly the same way as that of Theorem \ref{DiffusionTransfer}. Since we only need to prove the convergence in an interval $[0,1-\eps]$ for a given $\eps>0$, the fact that the coefficient of the noise term blows up at $t=1$ will not cause any problems.

Theorem \ref{limitPointProcessDecayingCase} can be proved the way (\ref{sineb}) was derived for the $\beta$-Hermite ensemble in \cite{VV}. The proof that we present below is not fully self-contained, we only highlight the main points of the arguments.

\begin{proof}[Proof of Theorem \ref{limitPointProcessDecayingCase}]
As an analogue of the continuous time phase function we define the discrete phase function $\varphi^{\la}_\ell $ with the identity $e^{i \varphi^{\la}_\ell }={X_\ell ^{\la}[1,1]}/{X_\ell ^{\la}[2,1]}$ and the relative phase function $\alpha^{\lambda}_\ell $ as $\varphi_\ell ^\la-\varphi_\ell ^0$. Note that $\varphi^{\lambda}_\ell $ can be defined as a continuous  function in $\la$ for any fixed $n$ which will make  $\alpha_\ell ^{\la}$ a well defined function. Equation (\ref{Xdef}) can be converted to a recursion for $e^{i \varphi_{\ell }^\la}$:
\begin{equation}\label{phirec}
e^{i \varphi^\la_\ell }=\mathcal{T}_\ell ^\la(e^{i \varphi^\la_{\ell -1}})
\end{equation}
where
\begin{equation}\mathcal{T}_\ell ^\la(v)=z^{2\ell} {\mathcal{X}^\la_\ell }(z^{-2\ell} v),\qquad
\mathcal{X}_\ell ^\la(\xi)=\frac{\xi(1+i \rho \eps_\ell/2 )+i \rho \eps_\ell/2 }{1-i \rho \eps_\ell/2 -i \rho \eps_\ell  \xi/2}.\label{phirec2}
\end{equation}
$\lambda$ is an eigenvalue if
\[
X_n^\la \vect{1}{0}=c Z^{-1} T^{-n} \vect{0}{1}=\frac{c i \rho}{2} \vect{-z^{n+1}}{\bar z^{n+1}}
\]
which is equivalent to $e^{i \varphi_n^\la}=-z^{2n+2}$.
The discrete version of the Sturm-Liouville theory implies that  the number of  eigenvalues in a given interval $[\lambda_1,\lambda_2]$ is given by the number of solutions of $e^{i x}=-z^{2n+2}$ with  $x\in[e^{ i \varphi_n^{\lambda_1}},e^{ i \varphi_n^{\lambda_2}}] $.
We can also count the eigenvalues using intermediate values of the phase function $\varphi^\la$. We define
$\tph^{\la}_k$ as a continuous function in $\lambda$ recursively using
\begin{equation*}
e^{i \tph_{0}^{\la}}=-z^{2n+2}, \qquad e^{i \tph_{k+1}^{\la}}=\left[\mathcal{T}_{n-k}^\la\right]^{-1}(e^{i \tph_{k}^{\la}}).
\end{equation*}
Then  the number of  eigenvalues in a given interval $[\lambda_1,\lambda_2]$ is given by
\begin{equation}\label{evcount}
\#\Big(\left[(\varphi_{n-k}^{\lambda_1}-\tph_{k}^{\lambda_1}),(\varphi_{n-k}^{\lambda_2}-\tph_{k}^{\lambda_2})\right]\cap 2\pi \Z\Big).
\end{equation}
The main steps of the theorem are as follows. The first step is straightforward from Theorem \ref{thmdecay1}.
\begin{step}
For every $0<\eps<1$ we have $\alpha^{\la}_{\lfloor n(1-\eps\rfloor)}\Rightarrow \alpha^{\la}(1-\eps)$ in the sense of finite dimensional distributions where $\alpha^\la(t)$ is the solution of SDE (\ref{SDErelphasedecay}).
\end{step}
The next step shows that the relative phase function $\alpha^\la$ cannot change too much from $n(1-\eps)$ to $n-k$.
\begin{step}\label{osc} There exist a constant $c>0$ depending only on $\sigma, \rho$ and $\La$ so that for every $|\la|\le\La$ and $k\le \eps n$ we have
\begin{equation}\label{alphabnd}
\Exp\left[(\alpha^\la_{\lfloor n(1-\eps\rfloor)}, \alpha^\la_{n-k})\wedge 1\right]\le c\of{
\Exp\,\, \textup{dist}(\alpha^\la_{\lfloor n(1-\eps\rfloor)},2\pi \Z)+\eps^{1/2}+n^{-1/2}+k^{-1}
}.
\end{equation}
\end{step}
The proof of Step \ref{osc} can be done in a similar way as in \cite{VV}. By analyzing the recursion (\ref{phirec}) we can get a precise estimate on $\Exp(\alpha_{\ell +1}^\la|\varphi_{\ell }^\la, \varphi_{\ell }^0)$. This can be turned into a Gronwall type estimate for $\textup{dist}(\alpha^\la_\ell ,\pi \Z)$ which leads to (\ref{alphabnd}). (See Sections 6.1 and 6.2 in \cite{VV} for details.) In order to estimate certain error terms one can take advantage of the fact that the rotation $z^{2\ell}$ in (\ref{phirec2}) has an averaging effect:
\[
|\sum_{\ell =\ell_1}^{\ell_2}  z^{2\ell} a_\ell |\le C (|a_{\ell _1}|+\sum_{\ell =\ell_1}^{\ell_2-1} |a_{\ell +1}-a_{\ell }| )
\]
We would like to note that this makes our case a lot easier to deal with than the  one in \cite{VV} where the dependence of the oscillation on $\ell$ was more complicated and needed much more involved estimates using harmonic analytic tools.

The next step shows that asymptotically in the formula (\ref{evcount}) only $\alpha_{n-k}^\la$ `matters'. The proof is analogue to the one presented in Sections 6.3 and 6.4 in \cite{VV}.
\begin{step}
If $k=k(n)\to \infty$ with $k/n\to 0$ then $\varphi_{n-k}^0$ converges to a uniform random variable on $[0,2\pi]$ modulo $2\pi$ in distribution. If $k$ is fixed then $\tilde \varphi_k^{\la}-\tilde \varphi_k^{0}\to 0$ in probability.
\end{step}

Now we have all the ingredients for the proof. Suppose that we want to show that for a given vector $(\la_1,\dots, \la_d)$ we have
\[
\textup{the number of ev's in } [0,\la_1], [0,\la_2],\dots, [0,\la_d]\Rightarrow (\alpha^{\la_1}(1), \dots, \alpha^{\la_d}(1)).
\]
By the previous statements we can find an appropriate sequence $k=k(n)\to \infty$ so that
\[
(\alpha_{n-k}^{\la_i}, i=1,\dots, d) \Rightarrow (\alpha^{\la_i}(1), i=1,\dots,d)\]
and $\tilde \varphi^{\la_i}_k-\tilde \varphi^0_k\to 0$ in probability for $i=1,\dots, d$. This means that if we apply formula (\ref{evcount}) with $\lambda_1=0$, $\lambda_2=\lambda_i$ then the
the length of the interval  will converge to $\alpha^{\la_i}(1)\in 2\pi \Z$ and the endpoint will become uniform modulo $2\pi$. Hence the number of lattice points in the $i^{th}$ interval will converge to $\frac1{2\pi} \alpha^{\la_i}(1)$ which proves the theorem along the found subsequence. But the argument can be repeated to find a converging sub-subsequence of any subsequence, and since we always get the same limit this shows the weak convergence along the original (full) sequence as well.
\end{proof}

\section{Appendix}\label{s:appendix}

The appendix contains the proof for the existence of unique analytic solutions for the discussed SDEs and a technical proposition about the convergence of discrete time Markov chains to stochastic differential equations.

\subsection*{Uniqueness and analyticity of the limiting SDE's}\label{s:sde}

\begin{prop}\label{SDEproperties}
The stochastic differential equations (\ref{LimitingTransfer}), (\ref{SDEphase}), (\ref{LimitingTransfer0}), (\ref{SDEphase0}), (\ref{LimitingTransferDecay}), (\ref{SDEphaseDecay}) all have unique
strong solutions which are analytic in $\lambda$. Moreover the solutions of  (\ref{SDEphase}),  (\ref{SDEphase0}) and (\ref{SDEphaseDecay}) are strictly increasing in $\lambda$ for any positive $t$.
\end{prop}
\begin{proof}
The coefficients of these SDE's are all uniformly Lipschitz so they have unique strong solutions for any  finite vector $\underline \lambda\in\Comp^d$. (Note that in the decaying case one may assume $t\in[0,1-\eps]$.)


In order to show that one can realize these solutions for all values of $\lambda$ together in a way that the dependence on $\lambda$ is analytic requires some extra work.

One possibility to deal with this problem is to use the
smooth dependence of the solution of an SDE on the initial
condition. We will use the following theorem which is a
slight modification of Theorem 40 in
\cite{Protter}.\begin{thm}[Protter, Theorem
40]\label{t:potter} Let $f_\alpha^i: \Real^d\to \Real, 1\le
i \le d, 0\le \alpha\le m$ be functions with locally
Lipschitz derivatives up to order $N$ for some $0\le N \le
\infty$. Then there exists a solution $X(t, \omega, x)$ to
\begin{equation}\label{qqq}
X_t^i=x_i+\int_0^t f^i_0(X_s) ds+\sum_{\alpha=1}^m \int_0^t f_\alpha^i (X_s) dB_s^\alpha, \qquad i=1,\dots, d
\end{equation}
which is $N$ times continuously differentiable in the open set $\{x: \zeta(x,\omega)>t\}$ where $\zeta$ is the explosion time of the solution. Moreover the respective derivatives in $x$ will satisfy the formal derivative of equation (\ref{qqq}).
\end{thm}
We can encode the dependence on $\lambda$ in (\ref{LimitingTransfer}) into dependence on initial condition by introducing  extra variables for $\re \lambda, \im \lambda$ and the extra equations $d \re \lambda=0, d \im \lambda=0$. Since for any fixed $\lambda$ the SDE has globally Lipschitz coefficients we will have $\zeta=\infty$. This shows that there exists a solution to (\ref{LimitingTransfer}) which is twice differentiable in the real variables $(x,y)=(\re \lambda, \im \lambda)$. The fact that we also get analyticity in $\lambda\in\Comp$ follows from the fact that the Cauchy-Riemann equations are satisfied. Indeed, at time $t=0$ we have $\partial_x X(t)=i \partial_y X(t)$ and it can be checked that the two processes satisfy the same SDE which means that the previous equation is preserved.

The same proof works for (\ref{SDEphase}), (\ref{LimitingTransfer0}), (\ref{SDEphase0}). In the case of (\ref{LimitingTransferDecay}) and (\ref{SDEphaseDecay})  the coefficients depend on $t$ as well, but introducing an extra variable for $t$ takes care of this (note that in this case $t\in [0,1)$).

To prove that the solution $\varphi^\lambda(t)$ of (\ref{SDEphase}) is increasing in $\lambda\in \Real$ we first compute the SDE for its derivative.
\begin{equation}\label{SDEphaseder}
d\of{\partial_\lambda \varphi^\lambda(t)}=dt+ \re \left[- i
\partial_\lambda \varphi^\lambda(t) e^{-i
\varphi^\lambda(t)} d\cW\right]=dt+ \,\partial_\lambda
\varphi^\lambda(t) \, \im \left[  e^{-i \varphi^\lambda(t)}
d\cW\right], \quad \partial_\lambda \varphi^\lambda(0)=0.
\end{equation}
For a given $\lambda$ the derivative solves the SDE
\begin{equation}\label{aSDE}
da=dt+\frac1{\sqrt{2}} a d\cB_\lambda, \quad a(0)=0
\end{equation}
with $d\cB_\lambda=\sqrt{2} \im \left[
e^{-i \varphi^\lambda(t)}, d\cW\right]$ and a simple
coupling argument shows that this is always positive for
$t>0$. (Actually, in this case one can even solve the SDE
explicitly.) Similar proof works for (\ref{SDEphase0}) and
(\ref{SDEphaseDecay}). Note that using the carousel
representation of the Section \ref{s:carousel} one can also
prove the monotonicity for (\ref{SDEphase})  and
(\ref{SDEphaseDecay}).
\end{proof}
\begin{cor}
For a given $\lambda\in \Real$ the distribution of $\partial_\lambda \varphi^\lambda(t)$ is the same as
\[
\int_{0}^t e^{-\frac{1}{\sqrt{2}} (\cB_s-\cB_t)+\frac14(s-t)} ds=e^{\frac{1}{\sqrt{2}} \cB_t-\frac14 t} \int_{0}^t e^{-\frac{1}{\sqrt{2}} \cB_s+\frac14 s} ds.
\]
\end{cor}
\begin{proof}
Using It\^o's formula it is straightforward to check that the process given in the statement of the corollary satisfies the SDE (\ref{aSDE}).
\end{proof}

\subsection*{Convergence of discrete time Markov processes to SDEs}
\begin{prop}\label{convergence} Fix $T>0$ and for each $n\geq 1$ consider a Markov chain
$$\of{X^n_\ell \in \Real^d, \ell=0\dots {\left \lfloor nT \right \rfloor }},$$
with $\Exp \norm{X^n_\ell}^2<\infty.$
For $x\in\Real^d$, let $Y^n_\ell (x) \in \Real^d$ be distributed as the increment of $X^n_{\ell +1}-X^n_\ell $ given
$X^n_\ell =x$. For $0\leq t\leq T$ and $x\in\Real^d$, let $b^n(t,x)\in\Real^d$ and $a^n(t,x)\in M_d^{\rm sym}(\Real)$ be defined by
\begin{equation*}
b^n(t,x):=n\Exp Y^n_{\left \lfloor nt \right \rfloor }(x),\qquad  a^n(t,x):=n\Exp Y^n_{\left \lfloor nt \right \rfloor }(x)Y^n_{\left \lfloor nT \right \rfloor }(x)^{\rm T}.\end{equation*}
We make the following assumptions.
\newline\noindent{\rm (1)} There are $C^2$ functions $a: [0,T]\times \Real^d\To M_d^{\rm sym}(\Real)$ and $b: [0,T]\times \Real^d\To M_d(\Real)$  such that for every $R<\infty$,
\begin{equation}
\sup_{0\leq t\leq T, \abs{x}\leq R}\norm{\int_0^t \of{a^n(s,x)-a(s,x)}ds }+\sup_{0\leq t\leq T, \abs{x}\leq R}\norm{\int_0^t \of{b^n(s,x)-b(s,x)}ds }\To 0.\label{e fia}
\end{equation}
\newline\noindent{\rm (2)} For every $R<\infty$ there is a constant $c_R<\infty$ such that
\begin{equation}
\norm{a^n(t,x)-a^n(t,y)}+\norm{b^n(t,x)-b^n(t,y)}\leq {c_{R}}\norm{x-y}, \label{e lip}
\end{equation}
for all $n\geq 1$, $t\in[0,T]$, $\norm{x}\leq R$ and $\norm{y}\leq R$. The same inequality holds for $a$ and $b$.
\newline\noindent{\rm (3)} For every $R<\infty$ there is a constant $d_R<\infty$ such that
\begin{equation}
\sup_{0\leq \ell\leq n, \norm{x}\leq R}\Exp[\norm{Y^n_\ell (x)}^3]\leq d_R n^{-3/2}. \label{e 3m}
\end{equation}
\newline\noindent{\rm (4)} The initial condition $X_0^n$ converges in distribution to $X_0$ with $\Exp\norm{X_0}^2<\infty.$
\newline Then $(X_{\left \lfloor nT \right \rfloor }^n, 0\leq t\leq T)$ converges  weakly in $D[0,T]$ to the unique solution of the SDE
\begin{equation}\label{SDESDE}
dX(t)=b(t,X(t))dt+g(t,X(t))d\cB(t),\qquad X(0)=X_0,
\end{equation}
where $\cB(t)$ is the $d$-dimensional Brownian motion and $g: [0,T]\times \Real^d\To M_d(\Real)$ is any  $C^2$ function with
$$g(t,x)g(t,x)^{\rm T}=a(t,x).$$
\end{prop}
Note: one can always take $g(t,x):=\of{a(t,x)}^{1/2}$ but it can be useful to make other choices for which $a(t,x)$ has sparser structure than $\of{a(t,x)}^{1/2}$ and the resulting SDE has a simpler noise term.
\begin{proof} This is Proposition 23 in \cite{VV} with two small changes:
there the supremum is for all $x$ in (1) and (3) and the functions $a, b$ are assumed to have bounded derivatives instead of being Lipschitz in $x$. The proof is very similar, but we include it for the sake of completeness.

Let $\|\cdot \|_\infty$ denote supremum norm on $[0,T]$.
For a two-parameter function $f$ and $x\in \Real$ let  $\IO$
denote the integral $\IO_{f,x}(t)= \int_0^t f(s,x)\,ds$. We
recycle this notation for a function $X:[0,T]\to \RR$ to
write  $\IO_{f,X}(t)= \int_0^t f(s,X(s))\,ds$.

Because of our assumptions on $a$ and $b$ the well-posedness of
the martingale problem follows from Theorem 5.3.7 of
\cite{EthierKurtz}  (see especially the remarks following the
proof), and even pathwise uniqueness holds. This means that (\ref{SDESDE}) has a solution $X$ with initial condition $X_0$ and this solution is unique in distribution.

Let $\tau^n_r=\inf\{t: |X^n(t)|\ge r\}$. The derivation of the convergence $X^n\cd X$ is based on Theorem 7.4.1 of \cite{EthierKurtz}, as well as Corollary 7.4.2 and its proof. These show that
 if the limiting SDE has a unique
solution (i.e.~the martingale problem is well-posed as it is in our case) and we have $X_0^n\cd X_0$ with
\begin{eqnarray}\label{e kuka}
\|(\IO_{b^n,X^n}-\IO_{b,X^n}) \mathbf{1}(t\le \tau^n_r) \|_\infty &\lip& 0,\\
\|(\IO_{a^n,X^n}-\IO_{a,X^n}) \mathbf{1}(t\le \tau^n_r) \|_\infty &\lip& 0,\nonumber
\end{eqnarray}
and
\begin{equation}
\textup{for every $\eps, r>0$}\qquad    \sup_{|x|\le r,\ell}
n \pr(|Y^n_\ell(x)|\ge \eps)\longrightarrow 0\label{e kuka3},
\end{equation}
then $X^n\cd X$. The theorem there only deals with the case
of time-independent coefficients, but adding time as an
extra coordinate extends  the results to the general case.

 Condititon  (\ref{e
kuka3}) follows  from the uniform third absolute
moment bounds (\ref{e 3m}) and Markov's inequality. Thus we only need to show (\ref{e
kuka}) as well as the analogous statement for $a$, for which the
proof is identical. We do this by bounding the successive
uniform-norm distances between
$$
\IO_{b^n,X^n},\quad \IO_{b^n,X^{n,L}}, \quad
\IO_{b,X^{n,L}},\quad \IO_{b,X^n},
$$
where $X^{n,L}_\ell= X^n_{K\lfloor \ell/K\rfloor} $ with $K=\lceil
nT/L \rceil$, and $X^{n,L}(t)=X^{n,L}_{\lfloor nt \rfloor}$. In
words, we divide $[0,\lfloor nT \rfloor]$ into $L$ roughly equal
intervals and then set $X_\ell^{n,L}$ to be constant on each
interval and equal to the first value of $X^{n}_\ell$ occurring
there.

If a function $f$ takes countably many values $f_i$, then
for any $h$ we have
$$
\|\IO_{h,f}  \mathbf{1}(t\le \tau_r^n)\|_\infty\le \sum_{i} \|\IO_{h,f_i}\mathbf{1}(t\le \tau_r^n)\|_\infty\
$$
Since $X^{n,L}$ takes at most $L$ values, we have
$$
 \|(\IO_{b^n,X^{n,L}}-\IO_{b,X^{n,L}})\mathbf{1}(t\le \tau^n_r)\|_\infty=\|\IO_{b^n-b,X^{n,L}}\mathbf{1}(t\le \tau^n_r)\|_\infty\le
 L \sup_{|x|\le r} \|\IO_{b^n-b,x}\|_\infty=Lo(1)
$$
by \eqref{e fia} where $o(1)$ is uniform in $L$ and refers to
$n\to \infty$.
From (\ref{e lip}), the other terms
satisfy
\begin{eqnarray*}
\|(\IO_{b^n,X^{n,L}}-\IO_{b^n,X^{n}})\mathbf{1}(t\le \tau^n_r)\|_\infty&\le&
T\|(b^n(\cdot,X^{n,L}(\cdot))-b^n(\cdot,X^{n}(\cdot)))\mathbf{1}(\cdot \le \tau^n_r)\|_\infty \\
&\le& c_r T \|X^n-X^{n,L}\|_\infty
\end{eqnarray*}
The same  holds with $b$ replacing $b^n$. It now
suffices to show that
\begin{equation}
\ev \|X^{n,L}-X^n\|_\infty = \ev \sup_\ell
|X^{n,L}_\ell-X^n_\ell|\le f(L)\label{e_L1}
\end{equation}
uniformly in $n$ where $f(L)\to 0$ as $L\to\infty$. The
left-hand side of (\ref{e_L1}) is bounded by
$$
\ev \sup_\ell |X^n_\ell-\frac1n\sum_{k={\lfloor
\ell/K\rfloor K}}^{\ell-1} b_n(X^n_\ell)-X^{n,L}_\ell|+ \ev
\sup_\ell |\frac1n\sum_{k={\lfloor \ell/K\rfloor K}}^\ell
b(X_\ell)|
$$
and the second quantity is bounded by $T\sup_{\ell,x}
|b^n_\ell(x)|/L$. The first quantity can be written as $\ev
M^*$ where
$$
M^*=\max_{i=0,\ldots,L-1} M_i^*, \qquad
M_i^*=\max_{\ell=0,\ldots, K-1} |M_{i,\ell}|,\qquad
 M_{i,\ell}=
X_{iK+\ell}-X_{iK}- \frac1n\sum_{k=0}^{\ell-1}
b^n(X_{iK+k}).
$$
Note that for each $i$, $M_{i,\ell}$ is a martingale. For any martingale
with $M_0=0$ we have
$$
\ev \max_{k\le n}|M_k|^3 \le c \ev \Big|
\sum_{k\le n} \ev[(M_k-M_{k-1})^2|\mathcal F_{k-1}]
\Big| ^{3/2}
\le c n^{3/2} \max_{k\le
n} \ev[|M_k-M_{k-1}|^3|\mathcal F_{k-1}].
$$
The first step is the  Burkholder-Davis-Gundy inequality   (see
\cite{Kallenberg}, Theorem 26.12) and the second step follows from
Jensen's inequality. Therefore \eqref{e 3m} implies
$$
\ev [|M_i^*|^3| \mathcal F_{iL}] \le c (n/L)^{3/2}
n^{\frac{-3}{2}} =c L^{\frac{-3}{2}},
$$
which gives the desired conclusion
$$
(\ev M^*)^3\le \ev (M^*)^3\le \ev \sum_{i=0}^{L-1}
(M_i^*)^3 \le c L^{\frac{-1}{2}}.
$$
Letting first $n\to\infty $ and then $L\to\infty$ gives \eqref{e_L1} and
\eqref{e kuka}.

\end{proof}
\noindent {\bf Acknowledgments.} This research is supported
by the NSERC discovery grant program and the Canada
Research Chair program (Vir\'ag). Valk\'o is supported by
the NSF Grant DMS-09-05820. We thank Hermann Schulz-Baldes
for references, Michael Aizenmann and Rowan Killip for many
interesting comments and discussions.


\bigskip

\noindent \sc Eugene Kritchevski. Department of
Mathematics, University of Toronto, Toronto ON~~M5S 2E4,
Canada. {\tt eugene.kritchevski@utoronto.ca}.

\bigskip
\noindent Benedek Valk\'o. Department of Mathematics,
University of Wisconsin Madison, WI 53705, USA. {\tt
valko@math.wisc.edu.}

\bigskip
\noindent B\'alint Vir\'ag. Departments of Mathematics and
Statistics. University of Toronto. Toronto ON~~M5S 2E4,
Canada. {\tt balint@math.toronto.edu}.

\end{document}